\documentclass[review]{elsarticle}

\usepackage{amssymb,amsmath,amsthm}

\usepackage{graphicx}

\usepackage{hyperref}

\usepackage{booktabs}

\usepackage{enumerate}

\newtheorem{theorem}{Theorem}[section]
\newtheorem{lemma}[theorem]{Lemma}

\theoremstyle{definition}
\newtheorem{definition}[theorem]{Definition}

\theoremstyle{remark}
\newtheorem{remark}[theorem]{Remark}

\numberwithin{equation}{section}

\def\p{\partial}
\def\tilde{\widetilde}
\def\hat{\widehat}
\def\bs{\boldsymbol}

\def\ub{\boldsymbol{u}}

\def\vb{\bs{v}}
\def\fb{\bs{f}}
\def\kb{\mathbf{k}}
\def\nb{\mathbf{n}}
\def\tb{\mathbf{t}}
\def\div{\mathrm{div}}
\def\curl{\mathrm{curl}}
\def\CE{\mathrm{CE}}
\def\VE{\mathrm{VE}}
\def\EC{\mathrm{EC}}
\def\EV{\mathrm{EV}}

\def\D{\mathcal{D}}
\def\R{\mathbb{R}}
\def\Ph{\mathcal{P}_h}
\def\Rh{\mathcal{R}_h}

\newcommand{\ra}[1]{\renewcommand{\arraystretch}{#1}}

\date{\today}

\usepackage{lineno,hyperref}
\modulolinenumbers[5]

\journal{Journal of Computational and Applied Mathematics}

\begin{document}

\begin{frontmatter}

\title{Stable and convergent approximation of two-dimensional vector
fields on unstructured meshes }

\author{Qingshan Chen}
\address{Department of Mathematical Sciences, Clemson University,
  Clemson, SC 29631, USA.}
\ead{qsc@clemson.edu}

\begin{abstract}
A new framework is proposed for analyzing staggered-grid finite
difference finite volume methods on unstructured meshes. The new
framework employs the concept of external approximation of function
spaces, and gauge convergence of numerical schemes through the
quantities of vorticity and divergence, instead of individual
derivatives of the velocity components. The construction of a stable
and convergent external 
approximation of a simple but relevant vector-valued function space is
demonstrated, and the new framework is applied to establish the
convergence of the MAC scheme for the incompressible Stokes problem on
unstructured meshes.
\end{abstract}

\begin{keyword}
staggered-grid \sep MAC \sep C-grid \sep incompressible Stokes \sep finite 
  difference \sep finite volume \sep unstructured meshes
\MSC 65N06 \sep 65N08 \sep 65N12 \sep 65N22 \sep 76D07
\end{keyword}

\end{frontmatter}

\linenumbers

\section{Introduction}


In simulations of physical systems, it is often advantageous to
stagger the vectorial variables with the scalar variables. The
resulting schemes are collectively called staggered-grid schemes. 
A classical example of staggered-grid scheme is
the Marker-and-Cell scheme (\cite{Harlow:1965jv}), also known
as the C-grid in the geoscience community (\cite{Arakawa:1977wr}), in
which the mass and other related variables are specified at cell
centers and the normal velocity components are specified at cell
edges. The MAC scheme is widely accepted as the method of choice for
incompressible flows ; see \cite{Wesseling:2001ci} for a review. Since its
introduction, it has also been argued that the scheme is suitable for
flows at all speeds; see the seminal papers by \cite{Harlow:1968ff} and
\cite{Harlow:1971iv}, as well as later developments by 
\cite{Bijl:1998gn,{Issa:1986gi}, {Karki:1989ez}, {vanderHeul:2003gd},
  {Wenneker:2002gt}}. For 
geophysical flows, the C-grid scheme has been shown to be superior in
resolving inertial-gravity dispersive relations; see the seminal paper
by \cite{Arakawa:1977wr}, and recent expositions on this
topic, \cite{Skamarock:2008hr,
  {Thuburn:2008bv},{Gassmann:2011bq},{Gassmann:2012ke}}.
In recent years, to take advantage of the growing power of
supercomputers, there has been a push to extend the C-grid scheme onto
unstructured meshes for complex problems on complex geometric
domains. In this regard, we mention the work by \cite{Thuburn:2009tb,
  {Ringler:2010io}, {Chen:2013bl}}.

In this work we concern ourselves with the analysis of staggered-grid
numerical schemes on unstructured meshes. 
Staggered-grid schemes are mostly constructed using the
finite difference (FD) or finite volume (FV) techniques, and
therefore, with regard to analysis,
they pose the same challenges that classical FD/FV schemes do, namely
the lack of variational formulations and the use of low-order
piecewise constant functions. Staggered-grid schemes on unstructured
meshes pose an extra challenge: 
the normal and/or tangential velocity components specified on the edges
may not align with the canonical directions of
the original vector field.
The last two decades have seen quite some efforts on this topic; the
theory for the MAC-scheme on structured meshes is fairly complete, at
least when classical fluid problems, such as
compressible/incompressible Stokes, are concerned. 
In 1975, roughly 10 years after the MAC scheme was introduced, Girault
(\cite{Girault:1976km}) proposed a finite element method on
``interlaced'' rectangular meshes for the stationary incompressible
Navier-Stokes problem; the method reduces to the classical MAC scheme
when the boundaries of the domain align with the mesh lines and a
special 4-point quadrature rule is used. First-order error estimates
were given for both the velocity and the pressure. 
Through the co-volume approach, Nicolaides and Wu
(\cite{Nicolaides:1996fm})  derives the a priori error estimates of the 
MAC scheme on rectangular meshes for the stationary two-dimensional
incompressible Navier-Stokes problem.
Kanschat (\cite{Kanschat:2008de}) shows that, on
rectangular meshes, the MAC scheme for incompressible flows is
algebraically equivalent to the divergence-conforming discontinuous
Galerkin method based on the lowest order Raviart-Thomas
elements. 
Eymard et.~al.~(\cite{Eymard:2010uj})
perform the stability and convergence analysis of the MAC scheme
for the two- and three-dimensional compressible Stokes problem on
rectangular meshes. Ch\'enier et.~al.~(\cite{Chenier:2012ty}) consider a variational
extension of the MAC to the full Navier Stokes equations on
semi-regular, non-conforming, and locally refined meshes. E and Liu (\cite{EL02})
analyze the MAC scheme for the two-dimensional time-dependent Navier-Stokes
equations, again, on rectangular meshes. 
The situation on unstructured meshes is quite different. The only work
known to us on this topic is Nicolaides (\cite{Nicolaides:1992vs}), who derives the
a priori error estimates for the MAC scheme for the incompressible
Stokes problem on unstructured meshes. 
We should point out that Chou (\cite{Chou:1997up}) derives the a priori error
estimates for MAC-like schemes for 
generalized Stokes on triangular meshes. But the schemes are
constructed by approximating both of the canonical velocity components
with piecewise linear Petrov-Galerkin elements, and thus are different
from the type of schemes considered here.

We aim to develop a new theoretical framework for
analyzing staggered-grid schemes for a wide range of
fluid problems. There are two essential ingredients to this new
framework. The first is the concept of external approximation of
function spaces, which was proposed by C\'ea (\cite{Cea:1964vy}), and
extensively used by Aubin (\cite{Aubin:1972wv}) and Temam
(\cite{Temam:2001wj}). 
This concept has been used recently by several authors to study the
convergence of non-staggered finite volume schemes
(\cite{Faure:2006ky,{Gie:2015tw}}). 
Formal definitions will be given in the next
section. Briefly speaking, external approximation adds an auxiliary
function space $F$ alongside the original function space $V$ and the
discrete function space $V_h$ (see Figure \ref{fig:external}). With the
aid of several mappings defined between these function spaces,
elements from $V_h$, which are often discontinuous, can now be compared
with elements from $V$ in the auxiliary space $F$.
The second ingredient of our new framework is the use of vorticity and
divergence to gauge the convergence of the numerical schemes. This is
a direct reflection of the fact that staggered-grid schemes are best
at mimicking the vorticity and/or divergence, but not the canonical
components of the velocity field, or any of its gradients in the
canonical directions. 

The framework is general enough to be applicable to different types of
staggered-grid schemes (MAC, co-volume, etc.), and potentially to a
wide range of fluid problems (compressible/incompressible Stokes,
shallow water equations, etc.)  
The goal of the current article is to present the analysis framework
and to apply it to the first case of interest, the classical
incompressible Stokes problem. The existence and uniqueness of a
discrete solution, and its convergence to the true solution are
established.  

After we have completed this work, we were made ware that very
similar results have been obtained in \cite{Eymard:2014bq}. But the
current work and the cited work differ in the approaches taken. We
utilize an external approximation framework for vector
fields for the convergence analysis, while \cite{Eymard:2014bq} rely 
on a strong reconstruction operator (to reconstruct the velocity field
from either the tangential or normal velocity components) and a
compactness result. They also prove the convergence for the pressure
field for one version of the MAC schemes, which comes as an extra
bonus of their approach. The same issue is not discussed in our work,
because the pressure field disappear in the variational form for the
problem. On the other hand, it appears that the results of
\cite{Eymard:2014bq} only apply to triangular-Delaunay meshes, while
ours apply to arbitrarily unstructured staggered grids. Due to these
differences, we are comfortable in publishing this work.

 The rest of the article is arranged as
follows. In Section 
\ref{sec:appr-vect-fields}, we recall the definitions of external
approximations, and present the framework for constructing and
analyzing external approximations of vector fields on unstructured
meshes. In Section \ref{sec:line-incompr-stok}, we apply the framework
to analyze the MAC scheme for the incompressible Stokes
problem. We finish in Section \ref{sec:concluding-remarks} with some
concluding remarks concerning the current work and future plans.

\section{Approximation of vector fields}\label{sec:appr-vect-fields} 
In the study of partial differential equations (PDEs) governing physical
systems, such as fluids, various vector-valued function spaces may
appear as the natural setting of the problems. These function spaces
usually differ in the level of the regularity and boundary
behaviors. We let $\Omega$ be a bounded and simply connected domain on
the two-dimensional plane with piecewise smooth boundaries, and in this
section, we consider the vector-valued function space 
\begin{equation}
\label{eq:173}
  V = H^\div_0(\Omega)\cap H^\curl(\Omega).
\end{equation}
Here $H^\div_0(\Omega)$ is a space of square-integrable vector-valued
functions whose divergence is also square-integrable, and whose normal
component vanishes on the boundary (see \cite{Girault:1986vn} for
details). Similarly, $H^\curl(\Omega)$ denotes a space of
square-integrable vector-valued functions whose curl is also
square-integrable. By \cite[Proposition 3.1]{Girault:1986vn}, the
space $V$ is algebraically and topologically included in the space
$H^1(\Omega)$, and in addition, the $H^1$-norm of functions from $V$
can be majorized by the $L^2$-norms of their divergence and
curl. Thus, $V$ is a Hilbert space with norm
\begin{equation}
 \label{eq:12}
   \|\ub\|_V^2 =   |\div\ub|_0^2 + |\curl\ub|_0^2.
 \end{equation}

In this section we present a discrete approximation to the function
space $V$. We choose to work on
this function space because it is quite general but still relevant in
the study of fluids. The homogeneous boundary condition on the normal
velocity component corresponds to no-flux boundary condition, which is
desirable for both viscous and inviscid fluids in closed domains. 
The discrete approximation to this function space appears to be the
most general setting where many of the discrete vector field theories
can be established, such as the Helmholtz decomposition theorem. These
results will be needed in dealing with specific problems, even though
the function spaces may differ.

The approximation that we are about to present is
stable and convergent. But before we present the discrete
approximations, we need to first recall the definitions of
approximations of function spaces, and the concepts of stability and
convergence. 


\subsection{Definitions}\label{sec:definitions}
Here we recall the definitions of external approximations
of normed spaces. Detailed expositions on this topic can be found in
\cite{Temam:1980wr,Cea:1964vy,Aubin:1972wv}. Let $V$ be a linear
function space with norm $\|\cdot\|$.

\begin{definition}\label{def:ext-approx}
  An external approximation of a normed space $V$ is a set consisting
  of
  \begin{enumerate}[\indent (a)]
  \item a normed space $F$ and an isomorphism $\Pi$ from $V$ into $F$;
    \item a family of triplets $\{V_h,\, \Ph,\,
      \Rh\}_{h\in\mathcal{H}}$, in which for each $h$, $V_h$ is a
      normed space, $\Ph$ a continuous linear mapping of $V_h$ into
      $F$, $\Rh$ a (possibly nonlinear) mapping of $V$ into $V_h$.
  \end{enumerate}
\end{definition}
The relations between the normed spaces and the operators in an
external approximation are shown in Figure \ref{fig:external}.
\begin{figure}[h]
  \centering
    \scalebox{0.7}{\includegraphics{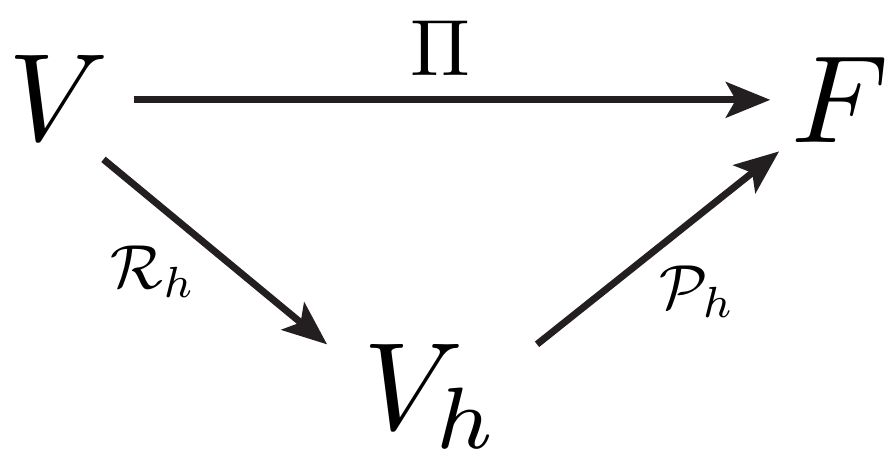}}
  \caption{External approximation}
  \label{fig:external}
\end{figure}

The restriction operators may be nonlinear, and, as a consequence, it
is not possible 
to define their norms. 
The prolongation operator $\Ph$ is linear, and its stability and the
stability of the external approximation are defined as follows.
\begin{definition}\label{def:stable-ext}
  The prolongation operators $\Ph$ are called stable if their norms
$$\|\Ph\| = \sup_{\stackrel{u_h\in V_h}{ \|u_h\|_h=1}} \|\Ph u_h\|$$ 
can be majorized independently of $h$. The external approximation of
$V$ is sable if the prolongation operators $\Ph$ are stable.
\end{definition}

\begin{definition}\label{def:conv-ext}
  An external approximation of a normed space $V$ is convergent if
  \begin{enumerate}[\indent ({C}1)]
  \item  for   all $u\in V$, 
    \begin{equation}
      \label{eq:81}
\lim_{h\rightarrow 0} \Ph \Rh u = \Pi u
    \end{equation}
in the strong topology of $F$;
\item a sequence $\Ph u_{h}$ converging to
  some element $\phi$ in the weak topology of $F$ implies that $\phi =
  \Pi u$ 
  for some $u\in V$.
  \end{enumerate}
\end{definition}

In practice, it may be difficult to
explicitly define the restriction operator $\Rh$ for every function in
$V$. As it turns out, $\Rh$ only needs to be specified for a dense
subspace $\mathcal{V}$. If condition (C1) holds for every
$\ub\in\mathcal{V}$, then the definition of $\Rh$ can be extended,
possibly in nonlinear fashion, to the whole space of $V$ so that the
condition holds for all $\ub\in V$. For a proof, the reader is
referred to \cite[Section
3.4]{Temam:1980wr}. 

\subsection{Specification of the mesh}\label{sec:specification-mesh}
\begin{figure}[h]
  \centering
  \includegraphics[width=\textwidth]{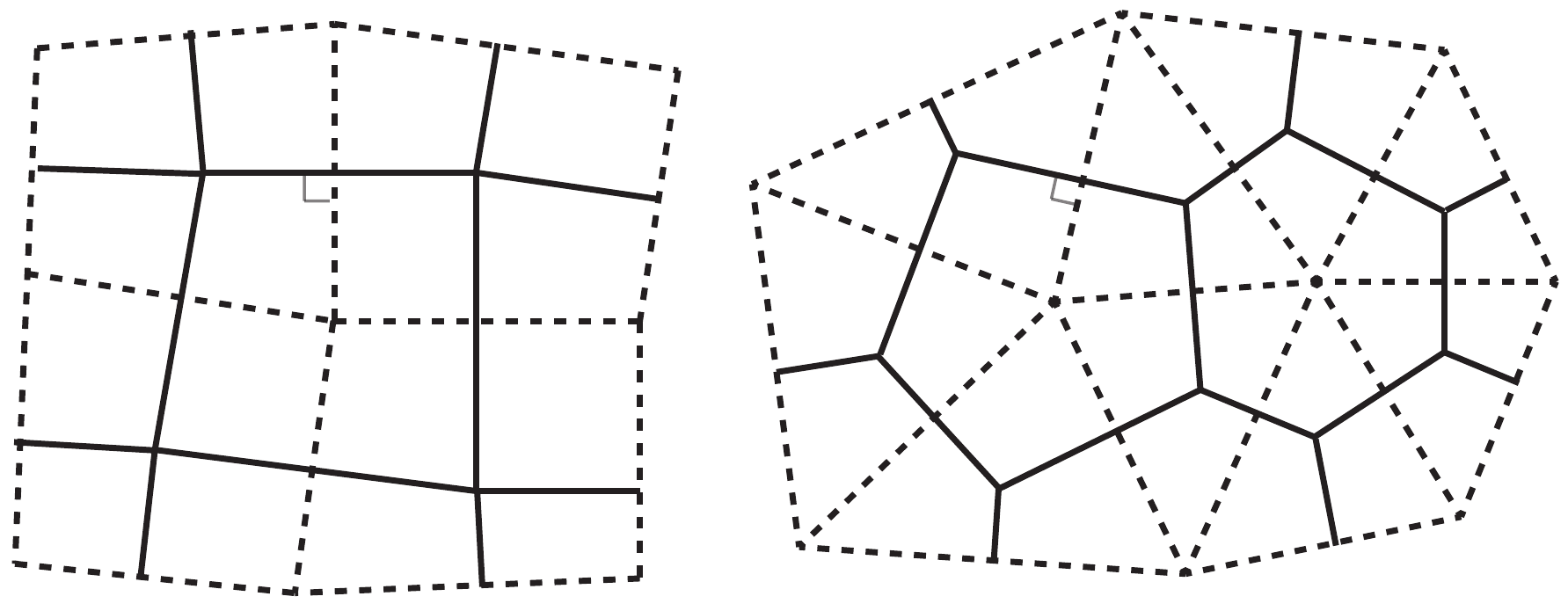}
  \caption{Examples of staggered grid. Left: a quadrilateral staggered
  grid. Right: a Delaunay-Voronoi staggered grid.}
  \label{fig:quad-dv}
\end{figure}

  
Our approximation of the function space is based on discrete meshes
that consist of polygons. To avoid potential technical issues with the boundary, we
shall assume that the domain $\Omega$ itself is polygonal. We make use
of a pair of staggered meshes, with one called primary and the other called
dual. The meshes consist of polygons, called cells, of arbitrary
shape, but conforming to the requirements to be specified. The centers
of the cells on the primary mesh are the vertices of the cells on the
dual mesh, and vice versa. The edges of the primary cells intersect
{\it orthogonally} with the edges of the dual cells. The line segments
of the boundary $\partial\Omega$ pass through the centers of the
primary cells that border the boundary. Thus the primary cells on the
boundary are only partially contained in the domain. 
Shown in Figure \ref{fig:quad-dv} are two common types of staggered
grids: a quadrilateral-quadrilateral staggered grid (left), and a
Delaunay-Voronoi tessellation (right). 

\begin{figure}[h]
  \centering
  \scalebox{0.65}{\includegraphics{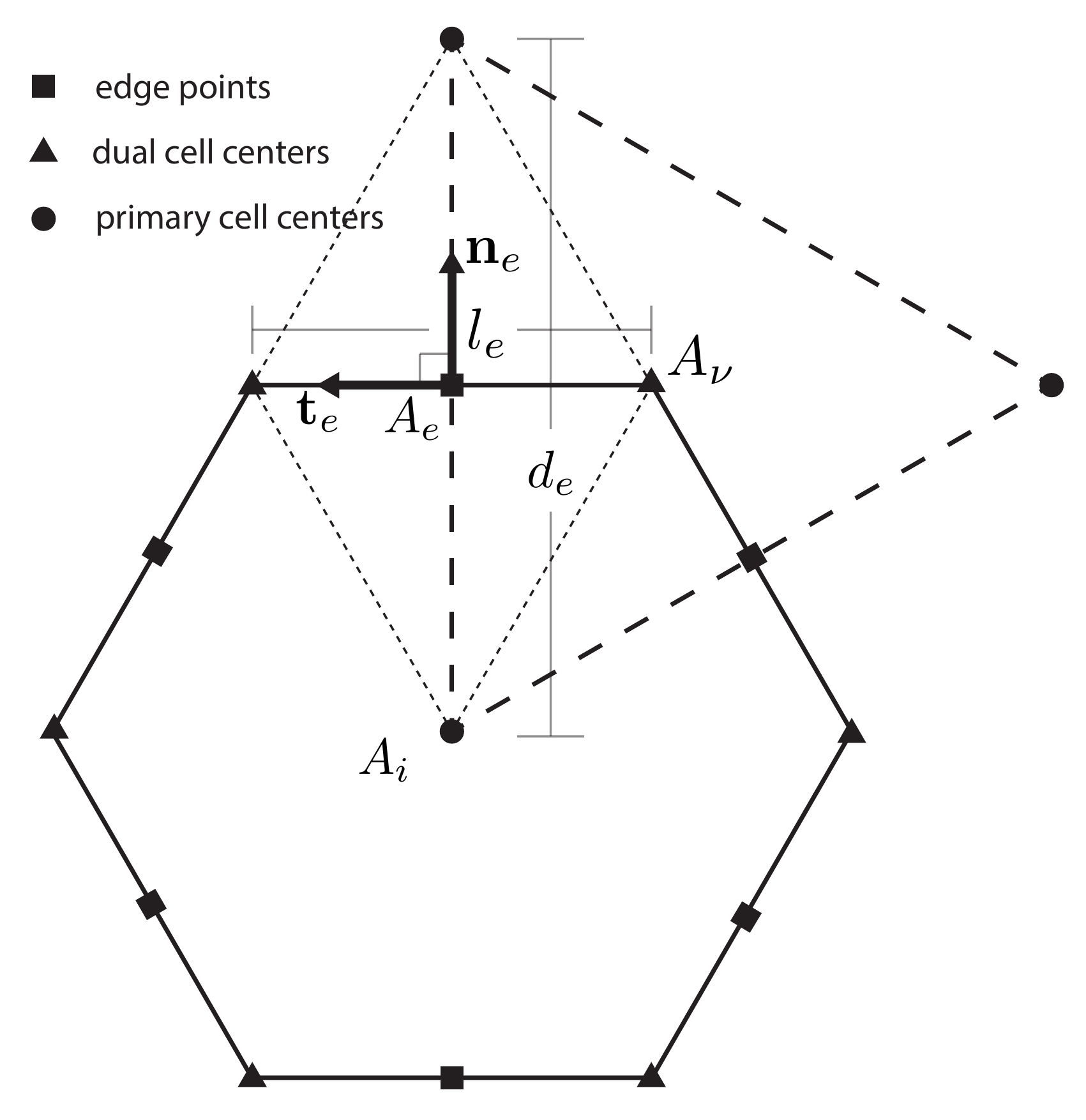}}
  \caption{Notations}
  \label{fig:notations}
\end{figure}

\begin{table}[h]
\centering
\ra{1.3}
\caption{Sets of elements defining the connectivity of an unstructured
  dual grid.}
\vspace{4mm}
\begin{tabular}{@{}ll@{}}\toprule
Set & Definition\\
\midrule
EC($i$) & Set of edges defining the boundary of primary cell $A_i$\\
VC($i$) & Set of dual cells that form the vertices primary cell $A_i$\\
CE($e$) & Set of primary cells boarding edge $e$\\
VE($e$) & Set of dual cells boarding edge $e$\\
CV($\nu$) & Set of primary cells that form vertices of dual cell $D_\nu$\\
EV($\nu$) & Set of edges that define the boundary of dual cell $D_\nu$\\
\bottomrule
\end{tabular}
\label{ta1}
\end{table}

In order to construct function spaces on this type of meshes, some
notations are in order, for which we follow the conventions made in
\cite{Ringler:2010io,Chen:2013bl}. As shown in the diagram in Figure
\ref{fig:notations}, the primary cells are denoted as $A_i,\, 1\leq
i\leq N_c + N_{cb}$, where $N_c$ denotes the number of cells that are
in the interior of the domain, and $N_{cb}$ the number of cells that
are on the boundary. We assume the cells are numbered so that $A_i$
with $1\leq i\leq N_c$ refer to interior cells.  The dual cells, which
all lie inside the domain, are denoted as $A_\nu,\,1\leq \nu\leq
N_v$. When no 
confusion should arise, we also use $A_i$ and $A_\nu$ to denote the
areas of the primary cells and dual cells, respectively. Each primary
cell edge corresponds to a distinct dual cell edge, and vice
versa. Thus the primary and dual cell edges share a common index $e,\,
1\leq e\leq N_e+N_{eb}$, where $N_e$ denotes the number of edge pairs
that lie entirely in the interior of the domain, and $N_{eb}$ the
number of edge pairs on the boundary, i.e., with dual cell edge
aligned with the boundary of the domain. Again, we assume that $1\le
e\le N_e$ refer to interior edges.
 Upon the edge pair $e$, the distance between the two
primary cell centers, which is also the length of the corresponding
dual cell edge, is denoted as $d_e$, while the distance between the
two dual cell centers, which is also the length of the corresponding
primary cell edge, is denoted as $l_e$. 
These two edges form the diagonals of a diamond-shaped region, whose
vertices consist of the two neighboring primary cell centers and the
two neighboring dual centers. The diamond-shaped region is also
indexed by $e$, and will be referred to as $A_e$.
The Euler formula for planar
graphs states that the number of primary cell centers $N_c + N_{cb}$, the
number of vertices (dual cell centers) $N_v$, and the number of
primary or dual cell edges $N_e + N_{eb}$ must satisfy the relation
\begin{equation}
  \label{eq:13}
  N_c + N_{cb} + N_v = N_e + N_{eb} + 1.
\end{equation}
The connectivity information of the unstructured staggered meshes is
provided by six {\it sets of elements} defined in Table \ref{ta1}. 

For each edge pair, a unit vector $\nb_e$, normal to the primary cell
edge, is specified. A second unit vector $\tb_e$ is defined as
\begin{equation}
  \label{eq:14}
  \tb_e = \kb\times\nb_e,
\end{equation}
with $\kb$ standing for the upward unit vector. Thus $\tb_e$ is
orthogonal to the dual cell edge, but tangent to the primary cell
edge, and points to the vertex on the left side of $\nb_e$. For each
edge $e$ and for each $i\in \CE(e)$ (the set of cells on edge $e$, see
Table \ref{ta1}), we define the direction indicator
\begin{equation}
\label{eq:15}
  n_{e,i} = \left\{
  \begin{aligned}
    1& & &\phantom{sssssss}\textrm{if }\nb_e\textrm{ points away from primary cell
    }A_i,\\
    -1& &  &\phantom{sssssss}\textrm{if }\nb_e\textrm{ points towards primary cell
    }A_i,\\
  \end{aligned}\right.
\end{equation}
and for each $\nu\in \VE(e)$,
\begin{equation}
\label{eq:115}
  t_{e,\nu} = \left\{
  \begin{aligned}
    1& & &\phantom{sssssss}\textrm{if }\tb_e\textrm{ points away from dual cell
    }A_\nu,\\
    -1& &  &\phantom{sssssss}\textrm{if }\tb_e\textrm{ points towards dual cell
    }A_\nu.\\
  \end{aligned}\right.
\end{equation}

For this study, we make the following regularity assumptions on the
meshes. We assume that the diamond-shaped region $A_e$ is actually
convex. In other words, the intersection point of each edge pair falls
inside each of the two edges. We also assume that the meshes are
quasi-uniform, in the sense that there exists $h>0$ such that, for 
each edge $e$,
\begin{equation}
  \label{eq:16}
  mh\leq l_e,\,d_e \leq Mh
\end{equation}
for some fixed constants $(m,\,M)$ that are independent of the
meshes. 
The staggered dual meshes are thus designated by $\mathcal{T}_h$.

\subsection{Discrete scalar fields}\label{sec:discr-scal-fields}
For each $1\leq i\leq N_c + N_{cb}$, let $\chi_i$ be the characteristic
function with support on cell $i$, that is,
\begin{equation}
  \label{eq:17}
  \chi_i(x) = \left\{
    \begin{aligned}
      &1 &\textrm{ if } x\in A_i,\\
      &0 & \textrm{ otherwise.}
    \end{aligned}\right.
\end{equation}
For each $1\leq \nu\leq N_v$, we let $\chi_\nu$ be the characteristic
function with support on dual cell $\nu$, that is,
\begin{equation}
\label{eq:18}
  \chi_\nu(x) = \left\{
    \begin{aligned}
      &1 &\textrm{ if } x\in A_\nu,\\
      &0 & \textrm{ otherwise.}
    \end{aligned}\right.
\end{equation}

We define $\Phi_h$ to be a space of scalar fields associated with the
primary mesh,
\begin{equation}
  \label{eq:19}
  \Phi_h = \left\{ \varphi_h = \sum_{i=1}^{N_c + N_{cb}} \varphi_i\chi_i, \textrm{
      with } \{\varphi_i\}_{i=1}^{N_c + N_{cb}}\in \mathbb{R}^{N_c + N_{cb}}\right\}.
\end{equation}
It is a Hilbert space endowed with the discrete $L^2$-norm
\begin{equation}
  \label{eq:20}
  \|\varphi_h\|_{\Phi_h}^2 \equiv |\varphi_h|_0^2 = \sum_{i=1}^{N_c + N_{cb}}A_i
  \varphi_i^2. 
\end{equation}

We define $\Psi_h$ to be a space of scalar fields associated with the
dual mesh,
\begin{equation}
\label{eq:21}
  \Psi_h = \left\{ \psi_h = \sum_{\nu=1}^{N_\nu} \psi_\nu\chi_\nu, \textrm{
      with } \{\psi_\nu\}_{\nu=1}^{N_\nu}\in \mathbb{R}^{N_\nu}\right\}.
\end{equation}
It is a Hilbert space endowed with the discrete $L^2$-norm
\begin{equation}
\label{eq:22}
  \|\psi_h\|_{\Psi_h}^2 \equiv |\psi_h|_0^2 = \sum_{\nu=1}^{N_\nu}A_\nu
  \psi_\nu^2. 
\end{equation}

Gradient operators can be defined on scalar fields from $\Phi_h$ and
$\Psi_h$, using the direction indicators $n_{e,i}$ and $t_{e,\nu}$,
respectively.  
On each edge $e$, the
discrete gradient operator on $\varphi_h\in\Phi_h$ is defined as
\begin{equation}
  [\nabla_h\varphi_h]_e = \dfrac{-1}{d_e}\sum_{i\in \CE(e)}\varphi_i n_{e,i},\label{eq:34}
 \end{equation}
and the skewed discrete gradient operator on $\psi_h\in\Psi_h$ is
defined as
\begin{equation}
  [\tilde\nabla_h^{\perp}\psi_h]_e = \dfrac{1}{l_e}\sum_{\nu\in
    \VE(e)}\psi_\nu t_{e,\nu}.\label{eq:35} 
 \end{equation}
The situation on the boundary requires some comments. With
each boundary edge, only one vertex is associated. Hence on a
boundary edge $e$, the definition \eqref{eq:35} can be written as
\begin{equation}
  [\tilde\nabla_h^{\perp}\psi_h]_{e\textrm{ on boundary}} =
  \dfrac{1}{l_e}\psi_\nu t_{e,\nu},\label{eq:40} 
 \end{equation} 
where $\nu$ is the single element in $\VE(e)$. This amounts to
implicitly requiring that $\psi_h$ vanishes on the boundary.
We let 
\begin{align} 
  &\nabla_h \varphi_h = \sum_{e=1}^{N_e + N_{eb}} [\nabla_h\varphi_h]_e \chi_e
  \nb_e,\label{eq:36}\\ 
  &\tilde\nabla_h^{\perp} \psi_h = \sum_{e=1}^{N_e + N_{eb}}
  [\tilde\nabla_h^{\perp}\psi_h]_e \chi_e \nb_e.\label{eq:37} 
\end{align}

With the gradient operators, semi-$H^1$ norms can be defined as
well. For $\varphi_h\in \Phi_h$, and $\psi_h\in \Psi_h$, we define
\begin{align}
  |\varphi_h|_{1,h} &\equiv |\nabla_h \varphi_h|_{0,h},  \label{eq:82}\\
 |\psi_h|_{1,h} &\equiv |\tilde\nabla^{\perp}_h \psi_h |_{0,h}.\label{eq:83}
\end{align}
These semi-$H^1$ norms can actually be taken as norms for the
corresponding function spaces, thanks to the discrete Poincar\'e
inequalities. We denote by $\dot{\Phi}_h$ the subspace of $\Phi_h$
that has zero average. 
\begin{lemma}[Discrete Poincar\'e inequalities for scalar
  fields]\label{lem:disc-poincare-scalar} 
For $\varphi_h\in\dot{\Phi}_h$ and $\psi_h\in\Psi_h$,
\begin{align}
  |\varphi_h|_{0,h} &\le C |\varphi_h|_{1,h},\label{eq:84}\\
|\psi_h|_{0,h} &\le C |\psi_h|_{1,h}.\label{eq:85}
\end{align}
In the above, $C$ stands for some generic constants that depend on the 
domain $\Omega$ only. 
\end{lemma}
The proofs of these  inequalities can be found in
\cite{Eymard:2000tt}. The proofs are quite technical, due to the lack
of a global Cartesian coordinate system. The main idea is to construct
the values of a scalar variable from its discrete derivatives along an
arbitrary but fixed direction. The dimension of the domain $\Omega$ along
that direction is finite, by assumption. For details, the reader is referred to
\cite{Eymard:2000tt}.

\subsection{Discrete vector fields}\label{sec:discr-vect-fields}
For each $1\leq e\leq N_e + N_{eb}$, we let $\chi_e$ be the characteristic
function with support on the diamond-shaped region $A_e$ (see Figure
\ref{fig:notations}, i.e. 
\begin{equation}
\label{eq:23}
  \chi_e(x) = \left\{
    \begin{aligned}
      &1 &\textrm{ if } x\in A_e,\\
      &0 & \textrm{ otherwise.}
    \end{aligned}\right.
\end{equation}
We define $V_h$ to be a space of discrete vector-fields that equal a
constant vector on each $A_e$. Specifically, 
\begin{equation}
  \label{eq:24}
  V_h = \left\{ u_h = \sum_{e=1}^{N_e + N_{eb}}u_e\chi_e \nb_e\right\}.
\end{equation}
We recall that $\nb_e$ is a unit vector normal to the primary
cell edge $e$. 

Around each primary cell $i$, a discrete divergence operator can be
defined, per the divergence theorem,
\begin{equation}
  \label{eq:26}
  \left[\nabla_h \cdot u_h\right]_i = \dfrac{1}{A_i}\sum_{e\in \EC(i)}u_e l_e n_{e,i}.
\end{equation}
It is worth noting that, on partial cells on the boundary, the
summation on the right-hand side only includes fluxes across the edges
that are inside the domain and the partial edges that intersect with
the boundary, and this amounts to imposing a no-flux
condition across the boundary.
It is clear from the definition \eqref{eq:26} that the image of the discrete divergence operator
$\nabla_h\cdot(\,)$ on each $u_h\in V_h$ is a scalar field in $\Phi_h$,
\begin{equation}
  \label{eq:27}
  \nabla_h\cdot u_h = \sum_{i=1}^{N_c + N_{cb}}\left[\nabla_h \cdot
    u_h\right]_i\chi_i    \quad\in\Phi_h,
\end{equation}
and the mapping is linear.
Around each dual cell $\nu$, a discrete curl operator can be defined,
per Stokes' theorem, 
\begin{equation}
\label{eq:28}
  \left[\tilde\nabla_h \times u_h\right]_\nu = \dfrac{-1}{A_\nu}\sum_{e\in \EV(\nu)}u_e d_e t_{e,\nu}.
\end{equation}
The tilde atop $\nabla$ signifies the involvement of the dual cells.
Thus, the image of the discrete curl operator
$\tilde\nabla_h\times(\,)$ on each $u_h\in V_h$ is a scalar field in $\Psi_h$,
\begin{equation}
\label{eq:29}
  \tilde\nabla_h\times u_h = \sum_{\nu=1}^{N_v}\left[\tilde\nabla_h \times
    u_h\right]_\nu\chi_\nu    \quad\in\Psi_h,
\end{equation}
and the mapping is linear.

$V_h$ is a finite dimensional Hilbert space under the 
discrete $L^2$-norm 
\begin{equation}
  \label{eq:30}
  |u_h|_{0,h}^2 \equiv \sum_{e=1}^{N_e + N_{eb}} A_e u_e^2.
\end{equation}
The discrete semi-$H^1$-norm on $V_h$ is defined as 
\begin{equation}
  \label{eq:31}
  |u_h|_{1,h}^2 \equiv |\nabla_h\cdot u_h|_{0,h}^2 + |\tilde\nabla_h\times u_h|_{0,h}^2.
\end{equation}
$V_h$ is a finite dimensional Hilbert space endowed with norm
\begin{equation}
  \label{eq:32}
\Vert u_h\|_{V_h}^2 = |u_h|_{0,h}^2 + |u_h|_{1,h}^2.  
\end{equation}
In fact, the semi-$H^1$ norm \eqref{eq:31} can also be taken as the
norm for $V_h$, thanks to a Poincar\'e-type inequality, which will be
presented after we state and prove a few basic properties for the discrete
divergence and curl operators. 

Given the definitions of the norm \eqref{eq:31} for $V_h$ and the norms
\eqref{eq:20} and \eqref{eq:22} for $\Phi_h$ and $\Psi_h$,
respectively, it is clear that the the discrete divergence operator
and the discrete curl operator
\begin{align}
  \nabla_h\cdot(\,) : \, &V_h\xrightarrow{\hspace{1.5cm}} \Phi_h,\label{eq:25}\\
  \tilde\nabla_h\times(\,) :\, &V_h\xrightarrow{\hspace{1.5cm}} \Psi_h,\label{eq:33}
\end{align}
are bounded linear operators.
From the definitions \eqref{eq:36} and \eqref{eq:37}, 
it is clear that $\nabla_h(\,)$ and $\tilde\nabla^\perp_h(\,)$ are
linear operators from $\Phi_h$ and $\Psi_h$, respectively, into $V_h$,
that is,
\begin{align}
  \nabla_h(\,) : \, &\Phi_h\xrightarrow{\hspace{1.5cm}} V_h,\label{eq:38}\\
  \tilde\nabla_h^\perp(\,) :\, &\Psi_h\xrightarrow{\hspace{1.5cm}} V_h.\label{eq:39}
\end{align}
They can be viewed as the ``adjoint operators'' of the discrete
divergence operator and the discrete curl operator, respectively,
thanks to the following discrete integration-by-parts formulae.
\begin{lemma}\label{lem:integ-by-parts}
  For $u_h\in V_h$, $\varphi_h\in\Phi_h$, and $\psi_h\in\Psi_h$, the
  following relations hold,
  \begin{align}
    \label{eq:41}
    \left( u_h,\,\nabla_h\varphi_h\right)_{0,h} &=
    -\dfrac{1}{2}\left(\nabla_h\cdot u_h,\,\varphi_h\right)_{0,h},\\
    \left( u_h,\,\tilde\nabla_h^\perp\psi_h\right)_{0,h} &=
    -\dfrac{1}{2}\left(\tilde\nabla_h\times u_h,\,\psi_h\right)_{0,h}.\label{eq:42}
  \end{align}
\end{lemma}
\begin{proof}
  We verify equation \eqref{eq:41} first. For arbitrary $u_h\in V_h$,
  and $\varphi_h\in\Phi_h$, by the definitions of the inner products
  and the discrete operators, we have
  \begin{align*}
    \left(u_h,\,\nabla_h\varphi_h\right)_{0,h} &=
    \sum_{e=1}^{N_e + N_{eb}}\sum_{i\in \CE(e)}\dfrac{-A_e}{d_e} u_e \varphi_i
    n_{e,i}\\
    &=\sum_{e=1}^{N_e + N_{eb}}\sum_{i\in \CE(e)}-\dfrac{1}{2}l_e u_e \varphi_i n_{e,i}.
  \end{align*}
We now switch the order of summations,
  \begin{align*}
    \left(u_h,\,\nabla_h\varphi_h\right)_{0,h} &=
-\dfrac{1}{2}\sum_{i=1}^{N_c + N_{cb}}\varphi_i \sum_{e\in \EC(i)}u_e l_e n_{e,i}\\
&=-\dfrac{1}{2}\sum_{i=1}^{N_c + N_{cb}}A_i \varphi_i\left(\dfrac{1}{A_i}
  \sum_{e\in \EC(i)}u_e l_e n_{e,i}\right).
  \end{align*}
From the definition of the discrete divergence operator (\ref{eq:27})
it follows that 
\begin{equation*}
 \left( u_h,\,\nabla_h\varphi_h\right)_{0,h} =
    -\dfrac{1}{2}\left(\nabla_h\cdot u_h,\,\varphi_h\right)_{0,h}.
\end{equation*}

To show (\ref{eq:42}), we again invoke the definitions of the inner
product and the discrete operator, and by simple calculations, we find
that
\begin{align*}
  \left(u_h,\,\tilde\nabla^\perp\psi_h \right)_{0,h} &=
  \sum_{e=1}^{N_e + N_{eb}} A_e u_e \left(\dfrac{1}{l_e} \sum_{\nu\in \VE(e)}
      \psi_\nu t_{e,\nu}\right)\\
&= \sum_{e=1}^{N_e + N_{eb}} \dfrac{1}{2}u_e d_e \sum_{\nu\in \VE(e)}\psi_\nu
t_{e,\nu}\\
&= \dfrac{1}{2} \sum_{e=1}^{N_e + N_{eb}} u_e d_e \sum_{\nu\in VE(e)}\psi_\nu t_{e,\nu}.
\end{align*}
Now we switch the order of summations,
\begin{align*}
  \left(u_h,\,\tilde\nabla^\perp\psi_h \right)_{0,h} &=
\dfrac{1}{2}\sum_{\nu=1}^{N_\nu}\psi_\nu \sum_{e\in \EV(\nu)} u_e d_e 
t_{e,\nu}\\
&= -\dfrac{1}{2} \sum_{\nu=1}^{N_\nu} A_\nu
\psi_\nu\left(\dfrac{-1}{A_\nu} \sum_{e\in \EV(\nu)} u_e d_e  
t_{e,\nu}\right). 
\end{align*}
From the definition of the discrete curl operator (\ref{eq:29}) on the
dual mesh, it follows that 
\begin{equation*}
      \left( u_h,\,\tilde\nabla_h^\perp\psi_h\right)_{0,h} =
    -\dfrac{1}{2}\left(\tilde\nabla_h\times u_h,\,\psi_h\right)_{0,h}.
\end{equation*}
\end{proof}

It is clear that equations (\ref{eq:41}) and (\ref{eq:42}) are, respectively, the
discrete versions of the integration-by-parts formulae
\begin{align}
  \int_\Omega\ub\cdot\nabla\varphi dx =&
  -\int_\Omega\nabla\cdot\ub\varphi dx, & &\forall \ub\in 
  H^\div_0(\Omega),\,\varphi\in H^1(\Omega),\label{eq:58}\\
  \int_\Omega\ub\cdot\nabla^\perp\psi dx =&
  -\int_\Omega\nabla\times\ub\psi dx, & &\forall \ub\in 
  H^\curl(\Omega),\,\psi\in H^1_0(\Omega).\label{eq:59}
\end{align}
The factor of one half in the discrete version stems from the fact
that the discrete vector fields $u_h$, $\nabla_h\varphi_h$, and
$\tilde\nabla^\perp\psi_h$ contain only the normal component (in the
direction of $\nb_e$). The no-flux boundary condition, required for
\eqref{eq:58}, is implied in 
the specification of the discrete divergence operator, and the
homogeneous boundary condition on the scalar field $\psi$, required
for \eqref{eq:59}, is implied in the
specification of the discrete skewed gradient operator
$\tilde\nabla^\perp$ on $\psi_h$. 
See the comments following definition \eqref{eq:26} and the comments
following \eqref{eq:40}. It is worth pointing out that the identity
\eqref{eq:42} is still valid if $u_h$ vanishes on the boundary and
$\psi_h$ is arbitrary. Indeed, we will encounter this situation in the
next section in dealing with the incompressible Stokes problem.


In two-dimension, it is well known that a non-divergent vector field
is the curl of a scalar field, and an irrotational vector field is the
gradient of a scalar field, and the set of non-divergent vector
functions and the set of irrotational vector functions form an
orthogonal decomposition of the $L^2(\Omega)\times L^2(\Omega)$
function space (\cite[Section 3]{Girault:1986vn}). We now establish
the discrete version of these results for the space $V_h$. 

\begin{lemma}\label{lem:nondivergent}
Assume that the domain $\Omega$ is simply connected.  For $u_h \in
V_h$, 
\begin{equation}
  \label{eq:92}
\nabla_h\cdot u_h = 0
\end{equation}
if and only if there 
  exists $\psi_h \in \Psi_h$ such that 
  \begin{equation}
    \label{eq:43}
    u_h = \tilde\nabla_h^\perp\psi_h.
  \end{equation}
\end{lemma}
\begin{proof}
  We first show sufficiency. Let $u_h$ be given by a scalar field
  $\psi_h\in \Psi_h$ via
  \begin{equation*}
    u_h = \tilde\nabla^\perp_h \psi_h.
  \end{equation*}
Then for an arbitrary cell $i$, 
\begin{align*}
  [\nabla_h\cdot u_h]_i =& \dfrac{1}{A_i} \sum_{e\in \EC(i)} u_e l_e
  n_{e,i}\\
=&\dfrac{1}{A_i} \sum_{e\in\EC(i)}\left(\dfrac{1}{l_e}\sum_{\nu\in
    \VE(e)} \psi_\nu t_{e,\nu}\right) l_e n_{e,i}\\
=& \dfrac{1}{A_i} \sum_{e\in \EC(i)}\sum_{\nu\in \VE(e)} \psi_\nu
t_{e,\nu} n_{e,i}.
\end{align*}
It is easy to verify that, surrounding cell $i$, each $\psi_\nu$
appears exactly twice, with opposite signs. Hence the summation is
zero, and \eqref{eq:92} is proven. 

For necessity, let $u_h\in V_h$ be a discrete vector field such that
(\ref{eq:92}) holds. For an arbitrary vertex, say $\nu=1$, we set
$\psi_\nu =0$, or any other constants. For a vertex $\nu$ that is connected
to vertex $1$ by a common edge, $\psi_\nu$ can be obtained by
integrating $u_e$ on the common edge. Specifically, $\psi_\nu$ on
neighboring vertices can be obtained through the relation 
\begin{equation}
  \label{eq:61}
  u_e l_e = \sum_{\nu\in \VE(e)} \psi_\nu t_{e,\nu}.
\end{equation}
It is obvious that, on edges that originate from vertex $1$, $u_h$ is
given by the formula (\ref{eq:43}).

\begin{figure}[h]
  \centering
  \includegraphics[width=8cm]{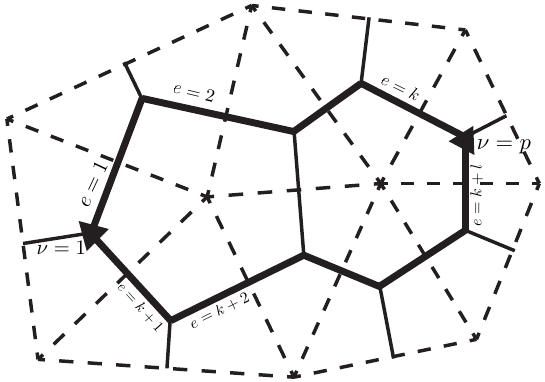}
  \caption{Closed path around primary cells}
  \label{fig:paths-primary}
\end{figure}

The integration can be further carried out to reach points that are
not directly connected to vertex $1$. By assumption, the domain
$\Omega$ is connected, and therefore every $\psi_\nu$ can be
determined this way. Since two arbitrarily given vertices can be
connected by multiple paths, we need to verify that results obtained
over different paths are consistent. To this end, we suppose that
vertex $p$ is connected to vertex 1 through two paths, the first
consisting of edges $1$, $2$, $\cdots$, $k$, and the second consisting
of edges $k+1$, $k+2$, $\cdots$, $k+l$. We also suppose that the
integration along the first path yields $\psi_p$, and the integration
along the second path yields $\tilde \psi_p$. We shall show that these
two results are identical. We note that the two paths form a closed
curve, and the surrounded region consists of primary cells, and no
holes (Figure 
\ref{fig:paths-primary}). By the assumption (\ref{eq:92}), the
net flux across the boundary of each individual cell is zero, and
therefore, it must 
also be zero across the boundary of the surrounded region, that is, 
\begin{equation}
  \label{eq:62}
  \sum_{e=1}^{k+l} u_e l_e n_e = 0,
\end{equation}
where $n_e$ is an indicator of the direction of the unit normal vector
$\nb_e$ with respect of the surrounded region, and it is defined as 
\begin{equation*}
  n_e = \left\{
    \begin{aligned}
      &1 & &\textrm{if $\nb_e$ points outward,}\\
      &-1 & &\textrm{if $\nb_e$ points inward.}
    \end{aligned}
\right.
\end{equation*}
Multiplying (\ref{eq:61}) by $n_e$  and summing over $1\leq e \leq k$,
we have
\begin{equation*}
  \sum_{e=1}^k u_e l_e n_e = \sum_{e=1}^k \sum_{e=1}^k \sum_{\nu\in
    \VE(e)} \psi_\nu t_{e,mu} n_e.
\end{equation*}
On the right-hand side, each $\psi_\nu$ except $\psi_1$ and $\psi_\p$
appears exactly twice, but with opposite signs. Hence
\begin{equation*}
  \sum_{e=1}^k u_e l_e n_e = -\psi_1 + \psi_p.
\end{equation*}
Similarly, along the second path, we have
\begin{equation*}
  \sum_{e=k+1}^{k+l}  u_e l_e n_e = \psi_1 - \tilde\psi_p.
\end{equation*}
It follows from (\ref{eq:62}) that 
\begin{equation*}
  \psi_p = \tilde\psi_p.
\end{equation*}
\end{proof}

\begin{lemma}\label{lem:irrotational}
Assume that the domain $\Omega$ is simply connected.  For $u_h \in
V_h$,  
  \begin{equation}
    \label{eq:63}
\tilde\nabla_h\times u_h = 0
  \end{equation}
 if and only if there
  exists $\varphi_h \in \Phi_h$ such that 
  \begin{equation}
\label{eq:44}
    u_h = \nabla_h\varphi_h.
  \end{equation}
\end{lemma}
\begin{proof}
  We first verify the sufficiency. We let $\varphi_h\in \Phi_h$ such
  that
  \begin{equation*}
    u_h = \nabla_h \varphi_h.
  \end{equation*}
Then for an arbitrary vertex $\nu$,
\begin{align*}
  [\tilde\nabla\times u_h]_\nu &= \dfrac{-1}{A_\nu}\sum_{e\in\EV(\nu)} u_e
  d_e t_{e,\nu}\\
&=\dfrac{-1}{A_\nu}\sum_{e\in\EV(\nu)}\dfrac{-1}{d_e}\left(\sum_{i\in\CE(e)}
  \varphi_i n_{e,i}\right) d_e t_{e,\nu}\\
&=\dfrac{1}{A_\nu} \sum_{e\in EV(\nu)} \sum_{i\in\CE(e)} \varphi_i
n_{e,i} t_{e,\nu}.
\end{align*}
We note that in the expression above concerning an arbitrary vertex
$\nu$, each $\varphi_i$ appears exactly twice, but with opposite
signs. Hence the summation vanishes for each $\nu$.

For necessity, we assume that $u_h\in V_h$, and (\ref{eq:63})
holds. We pick an arbitrary cell center, say cell $1$, and set 
\begin{equation*}
  \varphi_1 = 0.
\end{equation*}
Then we determine the values of the $\varphi$ at neighboring cell
centers by integrating $u_e$ along the dual cell edges. Specifically,
$\varphi_i$ at a neighboring cell center is obtained via
\begin{equation}
  \label{eq:64}
  u_e d_e = -\sum_{i\in\CE(e)}\varphi_i n_{e,i}.
\end{equation}
It is clear that the relation (\ref{eq:44}) holds along dual cell
edges originating from cell $1$.

\begin{figure}[h]
  \centering
 \includegraphics[width=8cm]{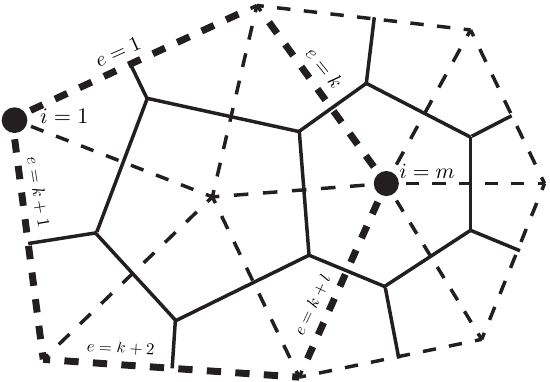}
  \caption{Two paths around dual cells}
  \label{fig:paths-dual-cells}
\end{figure}
The integration is then carried out to define $\varphi_i$'s at cell
centers not directly connected to cell $1$. The domain is connected,
and therefore each $\varphi_i$ can be defined this way. To ensure that
a discrete scalar field $\varphi_h$ is well-defined, we just need to
show that, for an arbitrary cell $m$, integrations along any two paths
yield the same value. Without loss of generality, we assume that the
first path consists of dual cell edges $1$, $2$, $\cdots$, $k$
(Figure \ref{fig:paths-dual-cells}), and
the integration yields $\varphi_m$, and the second path consists of
dual cell edges $k+1$, $k+2$, $\cdots$, $k+l$, and the integration
yields $\tilde\varphi_m$. These two paths form a closed curve, and the
enclosed region is made up of dual cells, and no holes. By the assumption
(\ref{eq:63}), the circulation around each dual cell is zero, and
therefore the circulation around the entire enclosed region is also
zero, that is,
\begin{equation}
  \label{eq:65}
  \sum_{e=1}^{k+l} u_e d_e t_e = 0,
\end{equation}
where $t_e$ is an indicator of the direction of the unit tangent
vector $\tb_e$ with respect to the enclosed region,
\begin{equation*}
  t_e = \left\{
    \begin{aligned}
      &1 & &\textrm{if $\tb_e$ points outward,}\\
      &-1 & &\textrm{if $\tb_e$ points inward.}
    \end{aligned}
\right.
\end{equation*}
Multiplying (\ref{eq:64}) by $t_e$ and summing over $1\le e \le k$, we
obtain
\begin{equation*}
  \sum_{e=1}^k u_e d_e t_e = -\sum_{e=1}^k\sum_{i\in\CE(e)} \varphi_i
  n_{e,i} t_e.
\end{equation*}
On the right-hand side, each $\varphi_i$ except $\varphi_1$ and
$\varphi_m$ appears exactly twice but with opposite signs, and hence
\begin{equation*}
  \sum_{e=1}^k u_e d_e t_e = -\varphi_1 + \varphi_m.
\end{equation*}
Similarly, multiplying (\ref{eq:64}) by $t_e$ and summing over $k+1\le
e \le k+l$, we find
\begin{equation*}
  \sum_{e=1}^k u_e d_e t_e = \varphi_1 - \tilde\varphi_m.
\end{equation*}
In view of (\ref{eq:65}), we conclude that 
\begin{equation*}
  \varphi_m = \tilde\varphi_m.
\end{equation*}
\end{proof}

\begin{lemma}\label{lem-orthogonal-decomp}
The space of discrete vector fields has the following orthogonal
decomposition
\begin{equation}
  \label{eq:45}
  V_h = \{ u_h\in V_h |\, \nabla_h\cdot u_h = 0\}\oplus\{u_h\in V_h
  |\, \tilde\nabla\times u_h = 0\}.
\end{equation}
\end{lemma}
\begin{proof}
We first show that the two sets are orthogonal. We let
$u_h,\,u^\#_h\in V_h$ such that $\nabla\cdot u_h = 0$ and
$\tilde\nabla\times u^\#_h=0$. Then by Lemma \ref{lem:nondivergent}, there
exists $\psi_h\in\Psi_h$ such that $u_h
=\tilde\nabla^\perp\psi_h$. Using the integration by parts formula
\eqref{eq:42}, we find that
\begin{equation*}
  (u_h,\,u^\#_h) = (\tilde\nabla^\perp \psi_h,\, u^\#_h) =
  -\dfrac{1}{2}(\psi_h,\,\tilde\nabla\times u^\#_h) = 0.
\end{equation*}
Thus $u_h$ and $u^\#_h$ are orthogonal.

We now show that each element of $V_h$ is the sum of an non-divergent
discrete vector field and an irrotational discrete vector field. In
view of Lemmas \ref{lem:nondivergent} 
and \ref{lem:irrotational}, this amounts to saying that there exist
$\varphi_h\in \Phi_h$ and $\psi_h\in \Psi_h$ such that
\begin{equation}
  \label{eq:66}
  u_h = \tilde\nabla^\perp\psi_h + \nabla_h\varphi_h.
\end{equation}
This single equation actually represents a system of equations
involving the normal velocity components $u_e$, $1\le e \le N_e + N_{eb}$, on
the edges, and the discrete scalar variable $\varphi_i$, $1\le i\le
N_c + N_{cb}$, at cell centers, and $\psi_\nu$, $1\le\nu\le N_v$, at cell
vertices. There are $N_c + N_{cb}+N_v$ unknowns. The system reads
\begin{equation}
  \label{eq:67}
  [\tilde\nabla^\perp\psi_h]_e + [\nabla_h \varphi_h]_e = u_e,\qquad
  1\le e\le N_e + N_{eb}.
\end{equation}
Hence there are $N_e + N_{eb}$ equations. By Euler's formula \eqref{eq:13}
there is one more unknown than the number of equations. This reflects
the fact that any $\varphi_h$ that satisfies \eqref{eq:66} will still
satisfy the equation after an addition of any constant. To make the
solution unique, we can impose an extra constraint requiring that the
area weighted average of $\varphi_h$ be zero, that is, 
\begin{equation}
  \label{eq:68}
  \int_\Omega \varphi_h dx = 0,\quad\textrm{or equivalently, }\quad\sum_{i=1}^{N_c + N_{cb}}
  A_i\varphi_i = 0.
\end{equation}
Equations \eqref{eq:67} and \eqref{eq:68} form a square linear
system. To show that this system has a unique solution for an
arbitrary $u_h\in V_h$, we only need to show that the homogeneous
system 
\begin{equation}
  \label{eq:69}
  \left\{
    \begin{aligned}
      &[\tilde\nabla^\perp\psi_h]_e + [\nabla_h \varphi_h]_e =0,\qquad
      1\le e\le N_e + N_{eb},\\
      &\sum_{i=1}^{N_c + N_{cb}}  A_i\varphi_i = 0,
    \end{aligned}\right.
\end{equation}
or equivalently,
\begin{equation}
  \label{eq:70}
  \left\{
    \begin{aligned}
      & \tilde\nabla^\perp\psi_h + \nabla_h \varphi_h = 0,\\
      &\int_\Omega\varphi_h dx = 0,
    \end{aligned}
\right.
\end{equation}
has only trivial solutions. It is clear that $\varphi_h=0$ and $\psi_h
= 0$ are solutions to the system \eqref{eq:70}. We let $\varphi_h$ and
$\psi_h$ be arbitrary discrete scalar functions that also satisfy the
system. Applying the discrete curl operator $\tilde\nabla_h\times(\,)$
to the first equation of \eqref{eq:70}, we obtain
\begin{equation}
  \label{eq:71}
  \tilde\nabla_h\times\tilde\nabla^\perp \psi_h = 0.
\end{equation}
Multiplying \eqref{eq:71} by $\psi_h$, and integrating by parts using
\eqref{eq:42}, we find that
\begin{equation}
  \label{eq:72}
  \left(\tilde\nabla^\perp\psi_h,\,\tilde\nabla^\perp\psi_h \right) = 0.
\end{equation}
Hence
\begin{equation}
  \label{eq:73}
  \tilde\nabla^\perp\psi_h = 0,\quad\textrm{or equivalently,
  }[\tilde\nabla^\perp \psi_h]_e = 0,\,\forall 1\le e\le N_e + N_{eb}.
\end{equation}
Noticing the definition \eqref{eq:40} of the skewed discrete gradient operator
$\tilde\nabla^\perp$ on the boundary, we conclude that 
\begin{equation}
  \label{eq:74}
  \psi_h = 0 \qquad\textrm{(i.e.~{}}\psi_\nu=0\quad\forall
  1\le \nu\le N_v).
\end{equation}
Applying the discrete divergence operator $\nabla\cdot(\,)$ to the
first equation of \eqref{eq:70} again, we obtain
\begin{equation}
  \label{eq:75}
  \nabla_h\cdot(\nabla_h\varphi_h) = 0.
\end{equation}
Multiplying \eqref{eq:75} by $\varphi_h$ and integrating by parts
using \eqref{eq:41}, we find that
\begin{equation}
  \label{eq:76}
  \left(\nabla_h\varphi_h,\,\nabla_h\varphi_h\right) = 0.
\end{equation}
Hence 
\begin{equation}
  \label{eq:77}
  \nabla_h\varphi_h = 0.
\end{equation}
Under the constraint $\eqref{eq:70}_2$, $\varphi_h$ must vanish
everywhere, that is,
\begin{equation}
\label{eq:78}
  \varphi_h = 0 \qquad\textrm{(i.e.~{}}\varphi_i=0\quad\forall
  1\le i\le N_c + N_{cb}).
\end{equation}
\end{proof}

A discrete Poincar\'e inequality concerning the $L^2$-norm
\eqref{eq:30} and the semi-$H^1$ norm \eqref{eq:31} of $V_h$ can
be established, 
which allows us to use the semi-$H^1$ norm as the norm for $V_h$.
\begin{lemma}[Discrete Poincar\'e inequality for vector
  fields]\label{lem:disc-poincare-vector} 
  For $u_h\in V_h$,
  \begin{equation}
    \label{eq:86}
    |u_h|_{0,h} \le C |u_h|_{1,h}.
  \end{equation}
\end{lemma}
\begin{proof}
  By Lemma \ref{lem-orthogonal-decomp}, there exist unique $\psi_h \in
  \Psi_h$ $\varphi_h\in\dot{\Phi}_h$ such that 
  \begin{equation}
    \label{eq:87}
    u_h = \tilde\nabla_h^\perp \psi_h + \nabla_h \varphi_h.
  \end{equation}
It is easy to see that $\psi_h$ and $\varphi_h$ satisfy the discrete
elliptic equations
\begin{align}
  \tilde\nabla_h\times\tilde\nabla_h^\perp \psi_h &=
  \tilde\nabla_h\times u_h,\label{eq:88}\\
  \nabla_h\cdot\nabla_h\varphi_h &= \nabla_h\cdot u_h.\label{eq:89} 
\end{align}
With the aid of the integration by parts formulae \eqref{eq:41} and
\eqref{eq:42}, and the discrete Poincar\'e inequalities \eqref{eq:84}
and \eqref{eq:85} for scalar fields, we derive the discrete analogues
to the classical energy bounds for elliptic problems,
\begin{align}
 & |\varphi_h|_{0,h} \le C |\varphi_h|_{1,h} \leq C|\nabla_h\cdot
  u_h|_{0,h},\label{eq:90}\\
&|\psi_h|_{0,h} \le C|\psi_h|_{1,h} \leq C|\tilde\nabla_h\times
u_h|_{0,h}.\label{eq:91} 
\end{align}
Here, $C$ stands for some generic constants that depend neither on the
function nor the mesh resolution $h$. Then, by the orthogonal decomposition
\eqref{eq:87}, and the estimates just obtained on $\varphi_h$ and $\psi_h$,
\begin{align*}
  |u_h|_{0,h}^2 &= |\tilde\nabla^\perp_h\psi_h|_{0,h}^2 +
  |\nabla_h\varphi_h|_{0,h}^2 \\
&= |\psi_h|_{1,h}^2 + |\varphi_h|_{1,h}^2\\
&\le C\left( |\tilde\nabla_h\times u_h|_{0,h}^2 + |\nabla_h\cdot
  u_h|_{0,h}^2\right)\\
&= C|u_h|_{1,h}^2.
\end{align*}
The claim is thus proven.
\end{proof}

\subsection{External approximation of $V$}\label{sec:extern-appr-v}
We recall that
\begin{equation*}
  V = H^1(\Omega)\cap H^\div_0(\Omega).
\end{equation*}
We let $F=H\equiv L^2(\Omega)\times L^2(\Omega)$, and for each $\ub\in
V$, we define
\begin{equation}
  \label{eq:46}
  \Pi \ub = (\nabla\times\ub,\,\nabla\cdot\ub)\quad \in F.
\end{equation}
It is clear that, thanks to the equality \eqref{eq:12},
\begin{equation}
  \label{eq:47}
  \Pi:\, V\xrightarrow{\hspace{1.5cm}} F
\end{equation}
is an isomorphism. For each $u_h\in V_h$, we define
\begin{equation}
  \label{eq:48}
  \Ph u_h = (\tilde\nabla_h\times u_h,\, \nabla_h\cdot u_h) \quad \in
  \Psi_h\times\Phi_h \subset F.
\end{equation}
Clearly $\Ph$ is a bounded linear operator from $V_h$ into $F$. We now
define the restriction operator $\Rh$. We only need to define $\Rh$ on
a dense subspace of $V$ (\cite{Temam:1980wr}). We let
\begin{equation}
  \label{eq:49}
  \mathcal{V} =\{ \ub\in C^\infty(\overline\Omega), \ub\cdot\nb = 0\textrm{
    on } \partial\Omega\}.
\end{equation}
Clearly $\mathcal{V}$ is a dense subspace of $V$. For each
$\ub\in\mathcal{V}$, we let $(\omega,\,\delta) = \Pi\ub$. We then
define their associated discrete scalar fields by
\begin{align}
  &\omega_h = \sum_{\nu=1}^{N_v}
  \omega_\nu\chi_\nu\quad\in\Psi_h,\label{eq:50}\\
  &\delta_h = \sum_{i=1}^{N_c + N_{cb}}
  \delta_i\chi_\nu\quad\in\Psi_h.\label{eq:51}
\end{align}
In the above, the discrete variables $\delta_i$ is set to be the
average of $\delta$ on primary cell $A_i$, i.e.
\begin{equation}
  \label{eq:52}
  \delta_i = 
  \dfrac{1}{|A_i|}\int_{A_i} \delta dx,
\end{equation}
so that, by the divergence theorem,
\begin{equation}
  \label{eq:57}
  \int_\Omega \delta_h dx = \int_\Omega \delta dx =
  \int_{\partial\Omega}\ub\cdot\nb ds = 0. 
\end{equation}
The discrete variables $\omega_\nu$ can be specified in various ways,
depending on the problem. For example, $\omega_\nu$ can be defined in
the same way that
$\delta_i$ is defined,
or it can simply be the value of $\omega$ at the center of
the dual cell $A_\nu$,
\begin{equation}
\label{eq:54}
  \omega_\nu = \omega(x_\nu),
\end{equation}
with $x_\nu$ being the coordinates of the dual cell center.
Then we let $u_h\in V_h$ be the discrete vector field from  that
satisfy 
\begin{equation}
  \label{eq:55}
  \left\{
    \begin{aligned}
      & \tilde\nabla_h\times u_h = \omega_h,\\
      & \nabla_h\cdot u_h = \delta_h.
    \end{aligned}\right.
\end{equation}
Assuming that the system \eqref{eq:55} is well-posed, i.e.~it has a
unique solution, we define such $u_h$ to be the image of $\Rh$ on
$\ub$, 
\begin{equation}
  \label{eq:56}
  \Rh\ub = u_h.
\end{equation}

We now show the well-posedness of the problem \eqref{eq:55}.
\begin{lemma}\label{lem:curl-div}
  For any $(\omega_h,\,\delta_h)\in \Psi_h\times\Phi_h$ satisfying
  $\int_\Omega \delta_h d x = 0$, the problem \eqref{eq:55} has a
  unique solution $u_h\in V_h$.
\end{lemma}
\begin{proof}
  We rewrite the system \eqref{eq:55} in terms of the discrete
  variables associated with the cell centers, cell vertices, and cell
  edges,
  \begin{equation}
    \label{eq:79}
    \left\{
      \begin{aligned}
        &\dfrac{-1}{A_\nu}\sum_{e\in\EV(\nu)} u_e d_e t_{e,\nu} =
        \omega_\nu,& &1\le\nu\le N_v,\\
        &\dfrac{1}{A_i} \sum_{e\in \EC(i)} u_e l_e n_{e,i} =
        \delta_i,& & 1\le i\le N_c + N_{cb}.
      \end{aligned}
\right.
  \end{equation}
It is clear that there are $N_e + N_{eb}$ unknowns, and $N_v+N_c + N_{cb}$
equations. According to the Euler formula \eqref{eq:13}, there is one
more equation than the number of unknowns. Hence the data on the
right-hand side of \eqref{eq:79} need to satisfy some constraint so
that they may belong to the range of the linear operator associated
with the system on the left-hand side. This constraint is provided by
the assumption 
$\int_\Omega \delta_h dx = 0$, because the integral of the left-hand
side of the second equation in $~\eqref{eq:79}$ always vanishes. Hence
for arbitrary 
$(\omega_h,\,\delta_h)\in \Psi_h\times\Phi_h$ satisfying the
constraint, the system \eqref{eq:79} or \eqref{eq:55} has a unique
solution if and only if the homogeneous system 
\begin{equation}
\label{eq:80}
  \left\{
    \begin{aligned}
      & \tilde\nabla_h\times u_h = 0,\\
      & \nabla_h\cdot u_h = 0,
    \end{aligned}\right.
\end{equation}
has only trivial solutions, which is evident from Lemma
\ref{lem-orthogonal-decomp}. 
\end{proof}

\begin{lemma}\label{lem:ext-stab-conv}
  The external approximation that comprises of the function space $F$,
  the isomorphic mapping $\Pi$, and the family of triplets
  $(V_h,\,\Rh,\,\Ph)_{h\in\mathcal{H}}$ is a stable and convergent
  approximation of $V$.
\end{lemma}
\begin{proof}
  By Definition \ref{def:stable-ext}, the external approximation is
  stable because the prolongation operator is stable. 

For convergence, we need to verify the two conditions specified in
Definition \ref{def:conv-ext}. We only need to verify \eqref{eq:81}
for $\ub\in\mathcal{V}$ (see \eqref{eq:49}). For an arbitrary $\ub\in
\mathcal{V}$, we let 
$(\omega,\delta) = \Pi\ub$, and, for each $1\le \nu\le N_v$ and $1\le
i\le N_c + N_{cb}$,
\begin{equation*}
  \omega_\nu = \overline\omega^{A_\nu},\qquad \delta_i =
  \overline{\delta}^{A_i}. 
\end{equation*}
Then by the definition \eqref{eq:56}  of the restriction operator
$\Rh$ and the definition \eqref{eq:48} of the prolongation operator
$\Ph$, 
\begin{equation*}
  \Ph \Rh \ub = \left(\sum_{\nu=1}^{N_v}\omega_\nu \chi_\nu,\,
    \sum_{i=1}^{N_c + N_{cb}}\delta_i\chi_i\right), 
\end{equation*}
and 
\begin{align*}
  \| \Ph \Rh \ub - \Pi \ub\|_F^2 =&
  \int_\Omega\left(\sum_{\nu=1}^{N_v}\omega_\nu\chi_\nu -
    \omega\right)^2 dx +
  \int_\Omega\left(\sum_{i=1}^{N_c + N_{cb}}\delta_i\chi_i - \delta\right)^2
  dx\\
=&  \sum_{\nu=1}^{N_v}\int_{A_\nu}(\omega_\nu -
    \omega)^2 dx +
  \sum_{i=1}^{N_c + N_{cb}}\int_{A_i}(\delta_i - \delta)^2
  dx\\
\le &\left(|\nabla\omega|^2_\infty + |\nabla\delta|^2_\infty\right)
\left(\sum_{\nu=1}^{N_v}\int_{A_\nu} 1dx + \sum_{i=1}^{N_c + N_{cb}}\int_{A_i}
  1dx\right) h^2\\
\le&2 |\Omega| \left(|\nabla\omega|^2_\infty +
  |\nabla\delta|^2_\infty\right) h^2.
\end{align*}
Therefore $\|\Ph \Rh\ub - \Pi \ub\|$ tends to zero as fast as the mesh
resolution $h$ tends to zero.

For the second condition of Definition \ref{def:conv-ext}, we note
that $(\omega,\delta)\in F$ is in the range of the linear operator
$\Pi$ if and only if $\int_\Omega \delta dx = 0$. Hence we only need
to verify that, if $(\omega,\delta)$ is the limit of the some sequence
$\Ph u_h$ in the weak topology of $F$, then $\int_\Omega \delta dx =
0$. The weak convergence of $\Ph u_h$ implies that
\begin{equation*}
  (\tilde\nabla_h\times u_h, \tilde\omega) + (\nabla_h\cdot
  u_h,\tilde\delta_h) \xrightarrow{\hspace{1cm}} (\omega,\tilde\omega)
  + (\delta,\tilde\delta),\qquad\forall (\tilde\omega,\tilde\delta)\in F.
\end{equation*}
If we set $\tilde\omega = 0$ and $\tilde\delta = 1$, then we have
\begin{equation*}
  0 = (\nabla_h\cdot u_h,\,1) \xrightarrow{\hspace{1cm}} \int_\Omega
  \delta dx.
\end{equation*}
Hence
\begin{equation*}
  \int_\Omega \delta dx = 0.
\end{equation*}
\end{proof}

\begin{figure}[h]
  \centering
  \includegraphics[width=6.5cm]{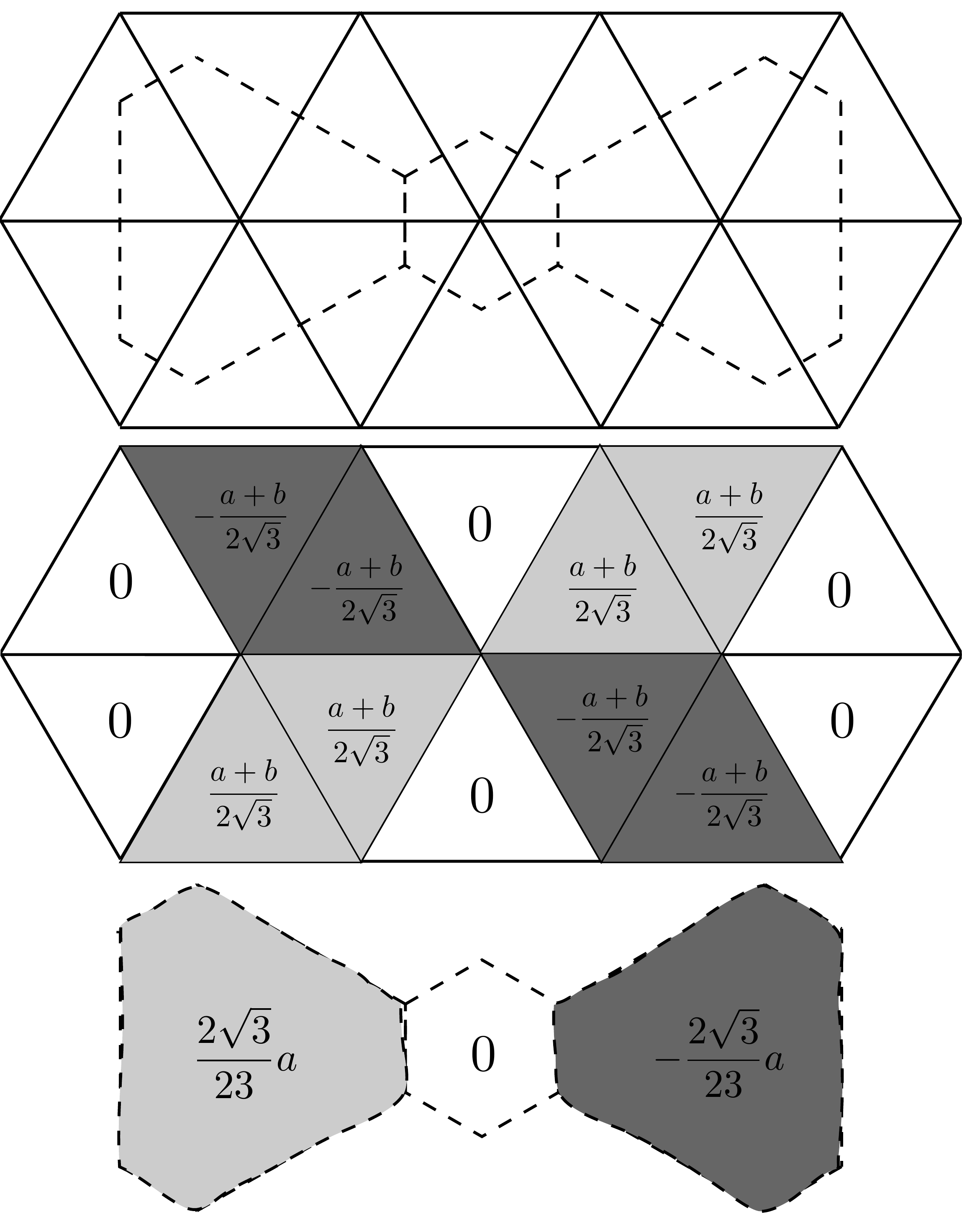}
  \caption{Defining $\Rh\ub$ by averaging the normal components of
    $\ub$ along the primary or dual cell edges can lead to
    inconsistently defined differential operators. Upper panel: the primary mesh
    consists of {\itshape equilateral} triangles, and the dual mesh consists of
    non-uniform but {\itshape equiangular} hexagons. The dual cell edge
    intersects the primary cell edge either at the mid-point or at the
  one-third point. Every two neighboring primary cell centers are
  equi-distant to the common edge that separates them. Middle panel: a
  non-diminishing scalar divergence field that results from averaging
  the normal components of $\ub = (ay,\,bx)$ along dual cell
  edges. Lower panel: a 
  non-diminishing scalar vorticity field that results from averaging
  the normal components of $\ub=(ax,\,by)$ along the primary cell edges. }
  \label{fig:trig-hex-third}
\end{figure}

\begin{remark}
It is tempting to define $\Rh \ub$ as the
average of the normal component of $\ub$ along one of the edges ($\l_e$
or $d_e$, see Figure \ref{fig:notations}). However, it is also
well-known in the finite volume literature (\cite{Eymard:2000tt,
  {Faure:2006ky}} that volume or area or length averaging leads to
inconsistently defined differential operators. In our terminologies,
condition C1 of Definition \ref{def:conv-ext} may be violated if the
restriction operator $\Rh$ is defined this way. One-dimensional
examples have been given in the two references just cited. Here we
give a two-dimensional example, in which the primary mesh consists of
{\itshape equilateral} triangles, and the dual mesh consists of non-uniform but
{\itshape equiangular} hexagons (see the upper panel of Figure
\ref{fig:trig-hex-third}).  Around each triangle, the dual cell edge
(dashed line) intersects the primary cell 
edge (solid line) either at the one-third point, or at the
mid-point. The triangular mesh
is Voronoi in the sense that every two neighboring cell centers are
equi-distant to the common edge between them, but the staggered meshes
are not the classical Delaunay-Voronoi meshes, because the roles of
the triangles and the roles of the hexagons are mutated. 
First, we discuss the case when the normal component of $\ub$ is
averaged along the dual cell edge (dashed lines). We let
$\ub=(ay,\,bx)$, where $a$ and $b$ are two arbitrary
constants. The analytic divergence $\delta = u_x + v_y$
vanishes 
everywhere. The discrete divergence $\delta_h = \nabla_h\cdot \Rh
\ub$ is a piecewise constant function on the primary mesh, and it can take three 
values, $0$, $(a+b)/2\sqrt{3}$ and $-(a+b)/2\sqrt{3}$. The
distribution pattern of these discrete values are shown in the middle
panel of Figure \ref{fig:trig-hex-third}. Clearly, for $a+b\ne 0$, $\delta_h$ does not
converge to $\delta$ in the $L^2$-norm as the mesh refines. It only
converges weakly. For 
the case of averaging along the primary cell edges (solid lines), we
let $\ub = (ax,\,by)$, again $a$, $b$ being arbitrary constants. The
analytic vorticity $\omega = v_x - u_y$ vanishes everywhere. But the
discrete vorticity $\omega_h = \nabla_h\times \Rh\ub$, which are
piecewise constant functions on the dual 
mesh, takes three possible values, $0$, $2\sqrt{3}a/23$ and
$-2\sqrt{3}a/23$, and the distribution pattern of these discrete
values is shown in the lower panel of Figure
\ref{fig:trig-hex-third}. It is clear that, for $a\ne 0$, the discrete
vorticity $\omega_h$ does not converge to the analytic vorticity
$\omega$ in the $L^2$-norm, as the mesh refines. It only converges
weakly. 
\end{remark}

\section{The linear incompressible Stokes problem}\label{sec:line-incompr-stok}

As in the previous section, we assume that $\Omega$ is an open,
bounded, and 
simply-connected domain with piece-wise smooth boundaries. The
incompressible Stokes problem reads 
\begin{equation}
\label{eq:93}
  \left\{{
  \begin{aligned}
    -\Delta\ub + \nabla p &= \bs{f},& &\Omega,\\ 
    \nabla\cdot\ub &= 0,& &\Omega,\\
    \ub &= 0, & &\partial\Omega.
  \end{aligned}}\right.
\end{equation}
 Using the vector identity
\begin{equation}
\label{eq:94}
  \Delta\ub = \nabla(\nabla\cdot\ub) + \nabla^\perp(\nabla\times\ub),  
\end{equation}
and the incompressibility condition, one derives the vorticity
formulation of the Stokes problem,
\begin{equation}
\label{eq:95}
  \left\{{
  \begin{aligned}
    -\nabla^\perp\nabla\times\ub + \nabla p &= \bs{f},& &\Omega,\\
    \nabla\cdot\ub &= 0,& &\Omega,\\
    \ub &= 0, & &\partial\Omega.
  \end{aligned}}\right.
\end{equation}
Here, $\nabla^\perp = \kb\times\nabla$, $\kb$ being the upward unit
vector, denotes the skewed gradient 
operator. We will use the vorticity formulation, since it 
highlights the role of vorticity, and seems
most suitable for staggered-grid discretization techniques. 

The natural functional space for the solution of
the Stokes problem \eqref{eq:93} is
$$V = \{\ub\in H^1_0(\Omega)\, | \, \nabla\cdot\ub =
0\textrm{ in } \Omega\}.$$
It is a Hilbert space under
the norm
\begin{equation}
  \label{eq:106}
  \|\ub\|_V^2 = |\nabla\times\ub|_0^2.
\end{equation}
By
integration by parts, we obtain the weak formulation of \eqref{eq:93}:
\begin{quote}
{\itshape  For each $\bs{f}\in L^2(\Omega)\times L^2(\Omega)$, find
  $\ub\in V$ such that 
  \begin{equation}
\label{eq:96}
    (\nabla\times\ub,\, \nabla\times\vb) = (\bs{f},\,\vb)\quad \forall
    \vb\in V.
  \end{equation}}
\end{quote}

It is clear that changes need to be made regarding
the discrete function space $V_h$ in order to accommodate the
incompressibility condition and the no-slip boundary condition, both
present in the new function space $V$. The previous definition \eqref{eq:24}
 of $V_h$ is changed to
\begin{equation}
\label{eq:97}
V_h = \left\{ u_h = \sum_{e = 1}^{N_e} u_e \chi_e \nb_e \,\biggr\vert\,\,
  \nabla_h\cdot u_h = 0.\right\} 
\end{equation}
The discrete divergence 
operator $\nabla_h\cdot(\,)$ and the discrete curl operator
$\tilde\nabla_h\times(\,)$ are defined as before, but with $V_h$
defined as in \eqref{eq:97},
the edges on the boundary has no effect on either the
divergence or the vorticity. 
The discrete Poincar\'e inequality still applies in this case, and
$V_h$ is a Hilbert space with the norm
\begin{equation}
  \label{eq:108}
  \|u_h\|_{V_h} \equiv |u_h|_{1,h}\equiv |\tilde\nabla_h\times u_h|_0.
\end{equation}

\subsection{External approximation of $V$}\label{sec:extern-appr-v-1} 
Since all functions in $V$ are divergence free, there is a one-to-one
correspondence between $V$ and $H^2_0(\Omega)$
(\cite{Girault:1986vn}).  
We let 
$$F =L^2(\Omega)\times L^2(\Omega).$$
For each $\ub\in V$, we let $\psi\in H^2_0(\Omega)$ and $\omega\in
L^2(\Omega)$ such that 
\begin{equation}
\label{eq:174}
\ub = \nabla^\perp \psi,\qquad \nabla\times\ub = \omega.
\end{equation}
Then we redefine the isomorphism $\Pi$ from $V$ into $F$ as
\begin{equation}
  \label{eq:107}
  \Pi\ub = (\psi,\,\omega)\in F,\qquad\forall \ub\in V.
\end{equation}
The space $F$ is endowed with the usual $L^2$-norm. In view of the
norm for $V$, it is clear that $\Pi$ is an isomorphism from $V$ into
$F$. It is important to note that the image of the operator $\Pi$ is
not the whole space $F$. It is a nowhere dense, closed subspace of the
latter. Given a vector function $(\psi,\,\omega)$ in $F$, there is no
known method to determine whether it is the image of some element
$\ub$ in $V$, other than solving an elliptic or biharmonic
equation. For this reason, we need a straightforward way of specifying
a restriction operator that is convergent. 
The specifications for the meshes are the same as in Section
\ref{sec:specification-mesh}. For this problem, we further 
assume that the primary cell edge and the dual cell edge bisect each
other 
in the interior of the domain. This requirement is stronger than what
is actually needed (see Remark \ref{rmk-h2}). 

Under the foregoing assumptions on the meshes, we now define the
restriction operator $\Rh$. We only need to define 
$\Rh$ for a dense subset of functions of $V$, and the definition can
then be extended to the whole space of $V$ according to a result in
\cite{Temam:1980wr}. We let 
\begin{equation}
\label{eq:170}
\mathcal{V} = \{ \ub \in \mathcal{D}(\Omega)\,|\, \nabla\cdot\ub = 0\},
\end{equation}
which is a dense subspace of $V$. For an arbitrary $\ub
\in\mathcal{V}$, there exists $\psi\in\D(\Omega)$ such that $\ub =
\nabla^\perp\psi$, and thus $\omega = \Delta\psi$. We then
define the 
associated discrete scalar field as
\begin{equation}
\label{eq:172}
\psi_h =  \sum_{\nu=1}^{N_v}\psi_\nu \chi_\nu,
\end{equation}
with
\begin{align}
  \psi_\nu &= \psi(x_\nu), & &\textrm{on interior dual cells,}\label{eq:112}\\
  \psi_\nu &= 0, & &\textrm{on dual cells on the  
    boundary,}  \label{eq:110}
\end{align} 
where $x_\nu$ is the coordinates for the center of dual cell
$\nu$. We note that the function $\psi$ has compact support on $\Omega$, and
therefore, if the grid resolution $h$ is fine enough, the
specification \eqref{eq:110} is consistent with \eqref{eq:112}. 
Finally, the
restriction operator $\Rh$ on $\ub\in\mathcal{V}$ is defined as
\begin{equation}
  \label{eq:113}
  \Rh\ub = \tilde\nabla_h^\perp\psi_h.
\end{equation}
That $\Rh\ub$ is divergence free in the discrete sense is guaranteed
by Lemma \ref{lem:nondivergent}. It vanishes on the boundary thanks
to the condition \eqref{eq:110} and the definition \eqref{eq:40} for
the skewed gradient operator on the boundary.

To define the prolongation operator $\Ph$, we note that, by the virtue
of Lemma \ref{lem:nondivergent}, every $u_h\in V_h$ is represented by
a scalar field $\psi\in\Psi_h$ via
\begin{equation}
  \label{eq:161}
  u_h = \tilde\nabla_h^\perp \psi_h.
\end{equation}
The prolongation operator $\Ph$ is defined as
\begin{equation}
  \label{eq:114}
  \Ph u_h = (\psi_h, \,\tilde\nabla_h\times u_h),\qquad\forall u_h \in V_h.
\end{equation}

The external approximation of $V$ consists of the mapping pair
$(F,\,\Pi)$ and the family of triplets
$\{V_h,\,\Rh,\,\Ph\}_{h\in\mathcal{H}}$. Concerning this approximation
we have the following claim.

\begin{theorem}\label{thm:stable-conv-h1_0}
  The external approximation that consists of the function space $F$,
  the isomorphic mapping $\Pi$, and the family of triplets $\{V_h,\,
  \Rh,\, \Ph\}_{h\in\mathcal{H}}$ is a stable and convergent
  approximation of $V$.
\end{theorem}

\begin{proof}
  According to Definition \ref{def:stable-ext}, the approximation is
  stable if the prolongation operators $\Ph$ are stable, which is
  evidently the case, in view of the specifications \eqref{eq:108} and
  \eqref{eq:114}, and the discrete Poincar\'e inequality. 

The approximation is convergent if Conditions (C1) and (C2) of Definition
\ref{def:conv-ext} are met. According to \cite{Temam:1980wr}, it is
only necessary to verify Condition (C1)  for the space $\mathcal{V}$.
 Let $\ub\in\mathcal{V}$ and $\Pi\ub = (\psi,\,\omega)\in F$, and let
 $\psi_h$ be defined as in \eqref{eq:112}-~\eqref{eq:110}. The (C1)
 condition is verified once we show that, as the grid resolution $h$ tends to
  zero, the discrete scalar field $\psi_h$ defined by 
  \eqref{eq:112}-~\eqref{eq:110}  converges strongly to $\psi$ in
  $L^2(\Omega)$, and 
  $\tilde\nabla_h\times\tilde\nabla_h^\perp\psi_h$ converges strongly to
  $\omega$ in $L^2(\Omega)$. 
The first claim can be easily verified by an application of the Taylor
series expansion of $\psi$. 
The second claim reflects the consistency of the discrete  Laplacian
operator on this mesh. Indeed, we note that, because the primary cell edge
$l_e$ and the dual cell edge $d_e$ bisect each other, 
\begin{equation}
[\tilde\nabla^\perp \psi_h]_e =
\overline{\left(\frac{\partial\psi}{\partial t_e}\right)}^{d_e} + O(h^2),
\end{equation}
where the overbar $\overline{\phantom{A}}^{d_e}$ denotes averaging along the dual
cell $d_e$. 
Applying the discrete curl operator to the above, we obtain
\begin{equation}
  \label{eq:132}
  [\tilde\nabla_h\times\tilde\nabla_h^\perp\psi_h]_\nu =
  \overline{(\Delta\psi)}^{A_\nu} + O(h),
\end{equation}
with the overbar $\overline{\phantom{A}}^{A_\nu}$ denoting averaging over the
dual cell $A_\nu$.
The claim can then be authenticated by application of the Taylor
series expansion to $\Delta\psi$.

For Condition (C2), we assume that a sequence $\{\Ph
u_h\}_{h\in\mathcal{H}}$, with $u_h\in V_h$, converges weakly to an
element $(\psi,\,\omega)\in F$, that is, as $h\longrightarrow 0$,
\begin{equation}
  \label{eq:111}
  \Ph u_h \rightharpoonup (\psi,\,\omega)\qquad\textrm{weakly in } F,
\end{equation}
which means that, according to definition \eqref{eq:114},
\begin{align}
 \psi_h &\rightharpoonup \psi\qquad\textrm{weakly in }
 L^2(\Omega),  \label{eq:162}\\ 
  \Delta_h \psi_h &\rightharpoonup \omega\qquad\textrm{weakly in }
  L^2(\Omega).\label{eq:163} 
\end{align}
Here, $\psi_h$ is a scalar field such that $u_h =
\tilde\nabla^\perp_h\psi_h$. 

We claim that 
\begin{align}
  &\psi\in H^2(\Omega), & & \label{eq:120}\\
  & \Delta\psi = \omega,& &\Omega\label{eq:121}\\
  &\psi = \dfrac{\partial\psi}{\partial n} = 0,&
  &\partial\Omega.\label{eq:122} 
\end{align}
Once these properties are verified, we set $\ub = \nabla^\perp
\psi$. It is clear then that $\ub\in V$ and $(\psi,\,\omega) = \Pi \ub$. 

We let 
\begin{align*}
  &\psi_h^\# & &= & &\textrm{extension of }\psi_h\textrm{ outside 
  }\Omega\textrm{ by zero},\\
  &\psi^\# & &= & &\textrm{extension of }\psi\textrm{ outside 
  }\Omega\textrm{ by zero},\\
  &(\Delta_h \psi_h)^\# & &= & &\textrm{extension of }\Delta_h \psi_h\textrm{ outside 
  }\Omega\textrm{ by zero},\\
  &\omega^\# & &= & &\textrm{extension of }\omega\textrm{ outside 
  }\Omega\textrm{ by zero}.\\
\end{align*}
The mesh $\mathcal{T}_h$ is also extended outside $\Omega$, and the
extended mesh $\mathcal{T}^\#_h$ satisfies the aforementioned requirements. 
We note that, thanks to the boundary conditions on $\psi_h$, 
\begin{equation}
  \label{eq:123}
  (\Delta_h\psi_h)^\# = \Delta_h \psi_h^\#.
\end{equation}
Convergences \eqref{eq:162} and \eqref{eq:163} imply that 
\begin{align}
   \psi^\#_h &\rightharpoonup \psi^\#,\qquad\textrm{weakly in }
  L^2(\mathbb{R}^2). \label{eq:124}\\
  \Delta_h \psi^\#_h &\rightharpoonup \omega^\#,\qquad\textrm{weakly 
    in } L^2(\mathbb{R}^2).\label{eq:125}  
\end{align}
We let $\psi'\in\mathcal{D}(\mathbb{R}^2)$, and let $\psi'_h =
\sum\psi'(x_\nu)\chi_\nu$.
Thanks to the compact support of $\psi'$, it can be shown in a
similar fashion as in the verification of the (C1) condition above that
\begin{align}
   \psi'_h &\longrightarrow \psi',\qquad\textrm{strongly in }
  L^2(\mathbb{R}^2).\label{eq:126}\\
  \Delta_h \psi'_h &\longrightarrow \Delta\psi',\qquad\textrm{strongly
    in } L^2(\mathbb{R}^2).\label{eq:127}
\end{align}
The following integration-by-parts formula holds for $\psi^\#_h$ and
$\psi'_h$, 
\begin{equation}
  \label{eq:128}
  (\psi^\#_h,\,\Delta_h\psi'_h) = (\Delta_h\psi^\#_h,\,\psi'_h).
\end{equation}
Thanks to the weak convergences \eqref{eq:162} and \eqref{eq:163} and
the strong convergences \eqref{eq:126} and \eqref{eq:127}, we can pass
to the limit in \eqref{eq:128} and obtain
\begin{equation}
  \label{eq:129}
  (\psi^\#,\, \Delta\psi') = (\omega^\#,\,\psi'),
\end{equation}
which implies that 
\begin{equation}
  \label{eq:130}
  \Delta\psi^\# = \omega^\#\qquad\textrm{in }\D'(\mathbb{R}^2).
\end{equation}
This relation, together with the fact that $\psi^\#\in L^2(\R^2)$
and $\omega^\#\in L^2(\R^2)$, implies that 
\begin{equation}
  \label{eq:131}
  \psi^\# \in H^2(\R^2).
\end{equation}
Restricted to the domain $\Omega$, \eqref{eq:131} and \eqref{eq:130}
imply \eqref{eq:120} and \eqref{eq:121}, respectively. The boundary
conditions \eqref{eq:122} for $\psi$ follow from the fact that
$\psi^\#\in H^2(\R^2)$, and $\psi^\#$ vanishes entirely outside
$\Omega$.  
\end{proof}

\begin{remark}\label{rmk-h2}
  The bisecting requirement on the meshes can be relaxed without
  affecting the convergence conclusion of Theorem
  \ref{thm:stable-conv-h1_0}. Specifically, the C1 condition for
  convergence relies on the second order accuracy of
  $[\tilde\nabla^\perp\psi_h]_e$ as an approximation to
  $\overline{(\partial\psi/\partial t_e)}^{d_e}$. The same order of
  accuracy can still be achieved if we allow the intersection of the
  primary cell edge $l_e$ and the dual cell edge $d_e$ to depart from
  their mid-points by no more than $O(h^2)$. 
\end{remark}

\subsection{Convergence of the MAC scheme}
The vorticity formulation \eqref{eq:95} is most suitable for
discretization on staggered grids, and this is the form that we will
use. In discretizing the system, it is important to ensure that
the external forcing $\fb$ is also discretized in a consistent way.
For the sake of the convergence proof later on, we discretize the
forcing term using its scalar stream and potential functions. For each
$\fb\in L^2(\Omega)\times L^2(\Omega)$, we let $\psi^f\in H_0^1(\Omega)$
and $\phi^f\in H^1(\Omega)/\mathbb{R}$ be such that 
\begin{equation}
\label{eq:168}
  \fb = \nabla^\perp \psi^f + \nabla \phi^f.
\end{equation}
By the famous Helmholtz decomposition theorem, the stream and
potential functions always exist and are unique, for each vector field
$\fb$ in $L^2(\Omega)\times L^2(\Omega)$ (see
\cite{Girault:1986vn}). The stream and potential functions are
discretized on the dual and primary meshes, respectively, by
averaging,
\begin{align}
  \psi^f_h &= \sum_{\nu=1}^{N_v} \psi^f_\nu\chi_\nu,& \textrm{with }
  \psi^f_\nu &= \overline{\psi^f}^{A_\nu},\label{eq:152}\\
  \phi^f_h &= \sum_{i=1}^{N_c+N_{cb}} \phi^f_i \chi_i,& \textrm{with }
  \phi^f_i &= \overline{\phi^f}^{A_i}.\label{eq:153}
\end{align}
Employing the technique of approximation by smooth functions and 
the Taylor's series expansion, we can show that the discrete scalar
fields converge to the corresponding continuous fields in the
$L^2$-norm, i.e.
\begin{align}
  \psi^f_h&\longrightarrow \psi^f & &\textrm{strongly in
  }L^2(\Omega),\label{eq:154} \\
\phi^f_h & \longrightarrow \phi^f & &\textrm{strongly in } L^2(\Omega).\label{eq:155}
\end{align}
With $\psi^f_h$ and $\phi^f_h$ defined as in \eqref{eq:152} and
\eqref{eq:153}, a discrete vector field can be specified, 
\begin{equation}
  \label{eq:109}
   f_h = \tilde\nabla_h^\perp\psi^f_h + \nabla_h \phi^f_h.
\end{equation}
We take $f_h$ as the discretization of the continuous vector forcing
field $\fb$.

The discrete problem can now be stated as follows. 
\begin{quote}
 {\itshape For each $\fb\in L^2(\Omega)\times L^2(\Omega)$, let $f_h$ be defined as
   in \eqref{eq:109}. Find $u_h\in V_h$ and $p_h \in \Phi_h$ such that 
   \begin{equation}
     \label{eq:134}
     -[\tilde\nabla_h^\perp\tilde\nabla_h\times u_h]_e + [\nabla_h
     p_h]_e = f_e,\qquad 1\leq e\leq N_e.
   \end{equation}
}
\end{quote}
The incompressibility condition and the homogeneous boundary
conditions on $u_h$ have been included in the specification of the
space $V_h$. It is important to note that equation \eqref{eq:134}
holds on interior edges only. On boundary edges, the computation of
$\tilde\nabla^\perp_h\tilde\nabla_h\times u_h$ will require boundary
conditions for $\tilde\nabla_h\times u_h$, which are not available a
priori.

As for the continuous problem, we multiply \eqref{eq:134} by
$v_h\in V_h$ and integrate by parts, and noticing that $v_h=0$ along
the boundary (see also Lemma \ref{lem:integ-by-parts}, and the remarks
following its proof), we obtain the
variational form of the numerical scheme,
\begin{equation}
  \label{eq:135}
  (\tilde\nabla_h\times u_h,\, \tilde\nabla_h \times v_h) = 2(f_h,\,v_h).
\end{equation}
The term involving the pressure $p_h$ vanishes because of the
incompressibility condition on $v_h$. The factor 2 on the right-hand
side of \eqref{eq:135}  results from the integration-by-parts
process. It can also be directly explained by the fact that the 
inner product on the right-hand side only involves the normal
components of the vector fields. For $u_h,\,v_h\in V_h$, we define the
bilinear form
\begin{equation}
  \label{eq:136}
  a_h(u_h,\,v_h) = (\tilde\nabla_h\times u_h,\, \tilde\nabla_h\times v_h).
\end{equation}
Then the variational form of the numerical scheme can be stated as
follows.
\begin{quote}
 {\itshape For each $\fb\in L^2(\Omega)\times L^2(\Omega)$, let $f_h$ be defined as
   in \eqref{eq:109}. Find $u_h\in V_h$ such that 
   \begin{equation}
\label{eq:137}
a_h(u_h,\,v_h) = 2(f_h,\,v_h),\qquad\forall v_h\in V_h.
   \end{equation}
}
\end{quote}

Given the norm \eqref{eq:108} on $V_h$, the bilinear form
$a_h(\cdot,\,\cdot)$ is coercive. Thus by Lax-Milgram theorem, for
every discrete vector field $f_h$, there exists a unique $u_h\in V_h$ such that
\eqref{eq:137} holds. Noticing that $v_h = 0$ on edges that intersects
with the boundary, we integrate the left-hand side of \eqref{eq:137}
by parts to obtain
\begin{equation}
\label{eq:167}
  \left( - \tilde\nabla_h^\perp \tilde\nabla_h \times u_h - f_h,\,
    v_h\right) = 0,\qquad \forall v_h\in V_h.
\end{equation}
We let $v_h = \tilde\nabla_h^\perp \psi_h$ for some $\psi_h\in
\Psi_h$ that vanishes on dual cells  that border the boundary. We
replace $v_h$ by $\tilde\nabla_h^\perp\psi_h$ in 
\eqref{eq:167}, and integrate by parts again to obtain
\begin{equation}
\label{eq:169}
  \left(\tilde\nabla_h\times( - \tilde\nabla_h^\perp \tilde\nabla_h
    \times u_h - f_h),\,     \psi_h\right) = 0.
\end{equation}
Thanks to the arbitrariness of $\psi_h$, equation \eqref{eq:169}
implies that 
\begin{equation}
\label{eq:171}
  \left[\tilde\nabla_h\times( - \tilde\nabla_h^\perp \tilde\nabla_h
    \times u_h - f_h)\right]_\nu = 0,\qquad\textrm{on interior dual
    cells.} 
\end{equation}
Following the same line of arguments as in the proof of Lemma
\ref{lem:irrotational}, we can show that there exists $p_h\in \Phi_h$,
unique up to a constant, such that 
\begin{equation}
  \label{eq:166}
 \left[ - \tilde\nabla_h^\perp \tilde\nabla_h
    \times u_h + \nabla_h p_h\right]_e =  f_e,\qquad\textrm{on
    interior edges.}
\end{equation}
Thus the pressure is recovered, and \eqref{eq:134} holds true.

\begin{remark}
The existence and uniqueness of a discrete solution to the system
\eqref{eq:134} can also be established from the point of view of a
square linear system. Indeed, 
in practice, the equations in \eqref{eq:134} are coupled with
the incompressibility constraints on $u_h$,
\begin{equation}
  \label{eq:150}
  [\nabla_h\cdot u_h]_i = 0,\qquad 1\leq i\leq N_c + N_{cb}.
\end{equation}
One of these equations is redundant, and should be dropped. Thus we
have $N_e + N_c + N_{cb} - 1 $ equations, for $N_e + N_c + N_{cb}$
unknowns ($u_e$'s with $1\leq e\leq N_e$ and $p_i$'s with $1\leq i\leq
N_c + N_{cb}$). There is one more unknown than the number of
equations, which is a reflection of the fact that if $p_h$ is a
solution of \eqref{eq:134}, then so is $p_h + c$ for any constant
$c$. To uniquely determine the pressure, we may impose an extra
constraint on $p_h$, such as
\begin{equation}
  \label{eq:151}
  \int_\Omega p_h dx = 0.
\end{equation}
The final system has $N_e + N_c + N_{cb}$ unknowns, and $N_e + N_c +
N_{cb}$ equations, and is a square linear system. For a finite
dimensional square linear system, uniqueness is equivalent to
solvability. Thus we can claim unique solvability for the system
\eqref{eq:134}, \eqref{eq:150} and \eqref{eq:151} once we show that
the only solutions 
corresponding to $f_h = 0$ is the trivial solution $u_h=0$ and
$p_h=0$. 
If $f_h=0$, then the only solution to \eqref{eq:137} is $u_h=0$,
thanks to the coercivity of the bilinear form $a_h(\cdot,\,\cdot)$. With
$u_h=0$ and $f_h=0$ in \eqref{eq:134}, we derive that $\nabla_h p_h =
0$, which means that $p_h$ is a constant over the entire domain. The
constraint \eqref{eq:151} implies that this constant must be
zero. Hence, for every $f_h$, the numerical scheme \eqref{eq:134} has
a unique solution.
\end{remark}

We now obtain a energy bound on the discrete solution in terms the
data. To this end, we set $v_h = u_h$ in \eqref{eq:137}, 
\begin{equation}
  \label{eq:156}
  |u_h|^2_{1,h} \equiv |\tilde\nabla_h\times u_h |_{0,h}^2 = 2(f_h, \, u_h).
\end{equation}
To estimate the right-hand side, we substitute \eqref{eq:109} for
$f_h$, and integrate by parts using formulae \eqref{eq:41} and
\eqref{eq:42} to obtain
\begin{equation}
  \label{eq:164}
  |u_h|_{1,h}^2 = -(\tilde\nabla_h\times u_h,\,\psi^f_h).
\end{equation}
The term involving $\phi^f_h$ has vanished due to the
incompressibility condition on $u_h$. An simple application of the
Cauchy-Schwarz inequality yields
\begin{equation}
  \label{eq:165}
  |u_h|_{1,h} \leq C|\psi^f_h|_0.
\end{equation}
Combining this equation with the fact that $\psi^f_h$
converges to $\psi^f$ as $h$ converges to zero, we
derive that
\begin{equation}
  \label{eq:158}
  |u_h|_{1,h} \leq C|\psi^f|_0 + K,
\end{equation}
where $C$ and $K$ are constants that are independent of $h$. 

The results are summarized in the
 following theorem.
 \begin{theorem}\label{thm:existence}
   For each $\fb\in L^2(\Omega)\times L^2(\Omega)$, let $f_h$ be defined as in
 \eqref{eq:109}. There exists a unique $u_h\in V_h$, and a
  $p_h\in\Phi_h$, unique up to a constant, such that \eqref{eq:134}
  holds. In addition, the discrete solution $u_h$ is bounded,
  \begin{equation}
    \label{eq:140}
    |u_h|_{1,h} \leq C|\psi^f|_0 + K,
  \end{equation}
where $C$ and $K$ are constants independent of the grid resolution
$h$. 
\end{theorem}

We next show that the discrete solution $u_h$ of \eqref{eq:137}
converges, and the limit is a solution of the continuous problem
\eqref{eq:96}. 
\begin{theorem}\label{thm:convergence}
  For each $\fb\in L^2(\Omega)\times L^2(\Omega)$, let $f_h$ be defined as in
  \eqref{eq:109}, and let $u_h$ be the unique solution of
  \eqref{eq:137}. Then there exists a unique $\ub\in V$ such that, as
  the grid resolution refines, 
  \begin{equation}
    \label{eq:141}
    \Ph u_h \longrightarrow \Pi\ub\qquad\textrm{strongly in } F,
  \end{equation}
and $\ub$ solves the variational problem \eqref{eq:96}.
\end{theorem}
\begin{proof}
  We first show that the discrete solutions $u_h$ converge. By the
  boundedness \eqref{eq:140} of $u_h$, there exists $(\psi,\,\omega) \in F$ and
  a subsequence $u_{h'}$ such that, as the grid resolution refines,
  \begin{equation}
    \label{eq:142}
    \mathcal{P}_{h'} u_{h'} \equiv (\psi_{h'},\,\tilde\nabla_{h'}\times u_{h'}) 
    \rightharpoonup (\psi,\,\omega)\quad\textrm{weakly in } F.
  \end{equation}
By the (C2) condition for a convergent approximation, there exists
$\ub\in V$ such that
\begin{equation*}
  (\psi,\,\omega) = \Pi\ub.
\end{equation*}

We next show that $\ub$ solves the continuous variational problem
\eqref{eq:96}. We let $\vb\in\mathcal{V}\subset V$, $v_{h'} = \mathcal{R}_{h'}
\vb$.
By the (C1) condition for a convergent
approximation, 
\begin{equation}
  \label{eq:143}
  \mathcal{P}_{h'} \mathcal{R}_{h'} \vb \equiv (\tilde\psi_{h'},\,\tilde\nabla_{h'}\times v_{h'}) \longrightarrow
  (\tilde\psi,\,\nabla\times\vb)\quad\textrm{strongly in } F.
\end{equation}
The discrete variational problem \eqref{eq:137} holds with
these $v_{h'}$ as the test functions,
\begin{equation}
  \label{eq:145}
  (\tilde\nabla_{h'}\times u_{h'},\,\tilde\nabla_{h'}\times v_{h'}) =
  2(f_{h'},\,v_{h'}). 
\end{equation}
Replacing the $f_{h'}$ on the right-hand side by
$\tilde\nabla_{h'}^\perp\psi^f_{h'} + \nabla_{h'}\phi^f_{h'}$, and
integrating by parts, we 
obtain
\begin{equation}
  \label{eq:146}
  (\tilde\nabla_{h'}\times u_{h'},\,\tilde\nabla_{h'}\times v_{h'}) =
  -(\psi^f_{h'},\,\tilde\nabla_{h'}\times v_{h'}).  
\end{equation}
In view of the convergences \eqref{eq:154}, \eqref{eq:142},
and \eqref{eq:143}, we pass to the limit in
\eqref{eq:146} by letting $h'\longrightarrow 0$, and obtain
\begin{equation}
  \label{eq:147}
  (\nabla\times\ub,\,\nabla\times\vb) = -(\psi^f,\,\nabla\times\vb).
\end{equation}
Integrating by parts again on the right-hand side yields
\begin{equation}
  \label{eq:148}
  (\nabla\times\ub,\,\nabla\times\vb) = (\fb,\,\vb). 
\end{equation}
Since $\mathcal{V}$ is dense in $V$, the above holds for every $\vb\in
V$, which confirms that $\ub$ is a solution of \eqref{eq:96}.

Finally, we show that the convergence \eqref{eq:142} holds for the
whole sequence $u_h$, and in the strong topology of $F$. The solution
$\ub$ of \eqref{eq:96} is necessarily unique. Then by a contradiction
argument, the convergence \eqref{eq:142} must hold for the entire
sequence of $u_h$. We now examine the difference $u_h - \Rh\ub$ in the
semi-$H^1$ norm of $V_h$.
\begin{align*}
  |u_h - \Rh \ub|_{1,h}^2 &= a_h(u_h - \Rh\ub,\, u_h-\Rh\ub)\\
&= a_h(u_h,\,u_h) + a_h(\Rh\ub,\,\Rh\ub) - 2 a_h(u_h,\,\Rh\ub)\\
&= 2(f_h,\,u_h) + (\tilde\nabla_h\times
\Rh\ub,\,\tilde\nabla_h\times \Rh\ub) - 4(f_h,\,\Rh\ub)  
\end{align*}
We let $\hat\psi_h\in \Psi_h$ 
be such that 
\begin{align*}
  \Rh\ub &= \tilde\nabla_h^\perp \hat\psi_h.\\
\end{align*}
Then, according to Theorem \ref{thm:stable-conv-h1_0},
\begin{equation}
  \label{eq:159}
\Ph \Rh\ub \equiv (\hat\psi_h,\,\tilde\nabla_h\times
\Rh\ub)\longrightarrow (\psi,\,\omega) = \Pi\ub\quad\textrm{strongly
  in } F.
\end{equation}
This relation, together with \eqref{eq:154} and \eqref{eq:142}, imply
that 
\begin{align*}
2(f_h,\,u_h) = -(\psi^f_h,\,\nabla_h\times u_h) &\longrightarrow
-(\psi^f_h,\,\nabla\times\ub) = (\nabla^\perp\psi^f,\,\ub) = (\fb,\,\ub),\\
4(f_h,\,\Rh\ub) = -2(\psi^f_h,\,\tilde\nabla_h\times \Rh\ub)
&\longrightarrow -2(\psi^f,\,\nabla\times\ub) =
2(\nabla^\perp\psi^f,\,\ub) = 2(\fb,\,\ub).
\end{align*}
The strong convergence of $\tilde\nabla_h\times \Rh\ub$ to
$\nabla\times\ub$ in $L^2(\Omega)$ also implies that 
\begin{align*}
  (\tilde\nabla_h\times
\Rh\ub,\,\tilde\nabla_h\times \Rh\ub) &\longrightarrow
(\nabla\times\ub,\,\nabla\times\ub).
\end{align*}
Hence we have
\begin{equation}
  |u_h - \Rh \ub|_{1,h}^2  \longrightarrow
  (\nabla\times\ub,\,\nabla\times\ub) - (\fb,\,\ub) = 0.\label{eq:149}
\end{equation}
The convergence \eqref{eq:141} follows from \eqref{eq:149} and the
following observation,
\begin{align*}
  |\Ph u_h - \Pi\ub|_0 &\leq |\Ph u_h - \Ph \Rh\ub|_0 + |\Ph \Rh\ub -
  \Pi\ub|_0\\
&\leq |\Ph|\cdot |u_h - \Rh\ub|_{1,h}  + |\Ph \Rh\ub - \Pi\ub|_0. 
\end{align*}
\end{proof}

\section{Concluding remarks}\label{sec:concluding-remarks}

In this article, we present a new framework for analyzing
staggered-grid schemes on unstructured meshes. The framework employs
the concept of external approximation to address the challenge that
comes with the use of piecewise constant functions in FD/FV
schemes. The framework uses vorticity and/or divergence to gauge the
convergence of the numerical schemes. Vorticity and divergence are two
fundamental quantities of fluid dynamics, and the performance of
numerical schemes in approximating these quantities is of great
interest, both theoretically and practically. In this work, we
demonstrate the construction and analysis of an external
approximation of the vector-valued function space
$H^\div_0(\Omega)\cap H^\curl(\Omega)$. The external approximation is
shown to be stable and convergent under the general orthogonal and
convex assumptions on the primary and dual meshes. We also apply the
framework to 
prove that the discrete solutions of the MAC scheme for the classical
incompressible Stokes problem on unstructured meshes converge to the
true solution, under an extra assumption that the primary cell edge and
the dual cell edge nearly bisect each other. More precisely, the
conclusion remains valid if the point of intersection between the
primary cell edge and the dual cell edge departs from their mid-points
by at most $O(h^2)$.

It is not known whether the just mentioned convergence result for the
Stokes problem still
holds without the bisection assumption at all. It would be a highly
desirable outcome if the assumption can be further
weakened so that only the primary cell edge bisects the dual cell edge
(or the other way around). In that case, the theoretical result will cover a
wider range of meshes, including the famous Delaunay-Voronoi
tessellations (\cite{Du:1999gs}). 

The current work is motivated by our study of the staggered-grid
schemes for the shallow water equations (\cite{Chen:2013bl},
\cite{Chen:QhglRipC}). So far, studies on this topic, including ours,
have largely been computational and experimental (\cite{Arakawa:1977wr},
\cite{Randall:1994uc}, \cite{Ringler:2010io},
\cite{LeRoux:2012ca}). Theoretical study, to 
establish the existence, uniqueness, and convergence of the discrete
solutions, is vital to ensure that the schemes perform under the most
general conditions. We believe that the framework presented in this
work is suitable for this task. So far, this framework has only been
applied to the incompressible Stokes problem. In order to apply the
framework to nonlinear problems, we envision that new results and new
techniques must be developed, such as the compactness of the discrete
function spaces. To facilitate such development, and to make progress
towards our ultimate goal, we will study a hierarchy of fluid models,
with increasing complexity and relevance to geophysical
flows, such as the stationary Navier-Stokes equations, the
compressible Stokes problem, etc. Work on these models will be
reported in future publications.

\section*{Acknowledgment}
The author thanks anonymous reviewers for their constructive comments
and their suggestions of references. 
The author also  warmly acknowledges helpful discussions with Lili Ju.
This work was in part supported by a grant from the Simons
Foundation (\#319070 to Qingshan Chen). 

\section*{References}
\bibliography{references}

\end{document}